\documentclass[10pt]{amsart}
\usepackage{amsrefs}
\usepackage{amssymb}
\usepackage[all]{xy}
\usepackage{tikz}
\usetikzlibrary{matrix,arrows}
\usepackage{hyperref}

\DeclareMathOperator{\Gal}{Gal}
\DeclareMathOperator{\Hom}{Hom}

\DeclareMathOperator{\Img}{Im}

\DeclareMathOperator{\Ker}{Ker}

\DeclareMathOperator{\res}{res}

\DeclareFontFamily{U}{wncy}{}
\DeclareFontShape{U}{wncy}{m}{n}{<->wncyr10}{}
\DeclareSymbolFont{mcy}{U}{wncy}{m}{n}
\DeclareMathSymbol{\Sha}{\mathord}{mcy}{"58}
\DeclareMathSymbol{\sha}{\mathord}{mcy}{"78}

\begin{document}

\newtheorem{thm}{Theorem}[section]
\newtheorem{cor}[thm]{Corollary}
\newtheorem{lem}[thm]{Lemma}
\newtheorem{prop}[thm]{Proposition}
\newtheorem{defin}[thm]{Definition}
\newtheorem{exam}[thm]{Example}
\newtheorem{examples}[thm]{Examples}
\newtheorem{rem}[thm]{Remark}
\newtheorem{case}{\sl Case}
\newtheorem{claim}{Claim}
\newtheorem{question}[thm]{Question}
\newtheorem{conj}[thm]{Conjecture}
\newtheorem*{notation}{Notation}
\swapnumbers
\newtheorem{rems}[thm]{Remarks}
\newtheorem*{acknowledgment}{Acknowledgment}
\newtheorem*{thmno}{Theorem}

\newtheorem{questions}[thm]{Questions}
\numberwithin{equation}{section}

\newcommand{\gr}{\mathrm{gr}}
\newcommand{\inv}{^{-1}}
\newcommand{\isom}{\cong}
\newcommand{\dbC}{\mathbb{C}}
\newcommand{\F}{\mathbb{F}}
\newcommand{\dbN}{\mathbb{N}}
\newcommand{\Q}{\mathbb{Q}}
\newcommand{\dbR}{\mathbb{R}}
\newcommand{\dbU}{\mathbb{U}}
\newcommand{\Z}{\mathbb{Z}}
\newcommand{\calG}{\mathcal{G}}
\newcommand{\K}{\mathbb{K}}


\newcommand{\hac}{\hat c}
\newcommand{\hatheta}{\hat\theta}

\title[2 pro-$p$ groups that are not absolute Galois groups]{Two families of pro-$p$ groups \\ that are not absolute Galois groups}
\author{Claudio Quadrelli}
\address{Department of Mathematics and Applications, University of Milano Bicocca, 20125 Milan, Italy EU}
\email{claudio.quadrelli@unimib.it}
\date{\today}

\begin{abstract}
Let $p$ be a prime.
We produce two new families of pro-$p$ groups which are not realizable as absolute Galois groups of fields.
To prove this we use the 1-smoothness property of absolute Galois pro-$p$ groups.
Moreover, we show in these families one has several pro-$p$ groups which may not be ruled out as absolute Galois groups employing the quadraticity of Galois cohomology (a consequence of Rost-Voevodsky Theorem), or the vanishing of Massey products in Galois cohomology.
\end{abstract}



\maketitle

\section{Introduction}
\label{sec:intro}

Throughout the paper $p$ will denote a prime number.
For a field $\K$, let $\bar\K_s$ and $\K(p)$ denote respectively the separable closure of $\K$, and the compositum of all finite Galois $p$-extensions of $\K$.
The {\sl maximal pro-$p$ Galois group of $\K$}, denoted by $G_{\K}(p)$, is the maximal pro-$p$ quotient of the absolute Galois group $\Gal(\bar\K_s/\K)$ of ${\K}$, and it coincides with the Galois group of the extension $\K(p)/\K$.
Describing maximal pro-$p$ Galois groups of fields among pro-$p$ groups is one of the most important --- and challenging --- problems in Galois theory (see, e.g., \cite[\S~3.12]{WDG}, \cite[\S~2.2]{birs}, and \cite[\S~1]{mrt}).
Already finding explicit examples of pro-$p$ groups which do not occur as maximal pro-$p$ Galois groups --- and thus also as absolute Galois groups (see Remark~\ref{rem:GKquot}) --- of fields is considered a remarkable achievement (see, e.g., \cites{BLMS,cem,mrt}). 
One of the oldest known obstructions for the realization of a pro-$p$ group as $G_{\K}(p)$ for some field $\K$ is given by the Artin-Scherier theorem: the only finite group realizable as $G_{\K}(p)$ for some field $\K$ is the cyclic group of order 2 (cf. \cite{becker}).
 
In this paper we produce new examples of torsion-free pro-$p$ groups which do not occur as maximal pro-$p$ Galois groups of fields containing a primitive $p$-th root of 1, and thus neither as absolute Galois groups.
Given two elements $x,y$ of a (pro-$p$) group $G$, let $[x,y]=x^{-1}y^{-1}xy$ denote the commutator between $x$ and $y$, and for $n\geq1$ set
\[
 [y,_nx]=[\ldots[[y,\underbrace{x],x],\ldots,x}_{n\text{ times}}].
\]

\begin{thm}\label{thm:main}
Let $p$ be a prime, and let $G$ be a finitely generated pro-$p$ group, with minimal pro-$p$ presentation
\[
 G=\langle\: x_1,\ldots,x_d\:\mid\: r_0,r_1,\ldots,r_\nu\:\rangle,\qquad d\geq3,\nu\geq0
\]
where either
\begin{itemize}
 \item[(i)] $r_0=x_1^q\cdot [x_1,_{n}\,x_2]\cdot [x_2,x_3]\cdots[x_{d-1},x_d]\cdot s$, with $d$ odd; or
 \item[(ii)] $r_0=x_1^q\cdot [x_1,_{n}\,x_2]\cdot [x_3,x_4]\cdots[x_{d-1},x_d]\cdot s$, with $d$ even; 
\end{itemize}
and in both cases $q=p^f$ with $f\in\{1,2,\ldots,\infty\}$ ($f\geq2$ if $p=2$), $n\geq2$, the relations $r_1,\ldots, r_\nu$ lie in the closed commutator subgroup $S':=\mathrm{cl}([S,S])$, and $s$ lies in the closed subgroup $\mathrm{cl}([S',S'])$, where 
$S$ denotes the closed subgroup of $G$ generated by $x_3,\ldots,x_d$.

Then $G$ does not occur as the maximal pro-$p$ Galois group of a field containing a primitive $p$-th root of 1, and thus neither as the absolute Galois group of a field.
\end{thm}

To prove Theorem~\ref{thm:main}, we employ the following formal version of {\sl Hilbert 90} for pro-$p$ groups, introduced in \cite[\S~14]{dcf:lift} and studied in \cites{eq:kummer,cq:1smooth,qw:cyc}.
A pro-$p$ group $G$ is said to be {\sl 1-smooth} if it may be endowed with a continuous $G$-module $M$, isomorphic to $\Z_p$ as an abelian pro-$p$ group, such that for every open subgroup $U$ of $G$ the canonical map 
 \[
  H^1\left(U,M_U/p^nM_U\right)\longrightarrow H^1\left(U,M_U/pM_U\right),
\]
induced by the epimorphism of continuous $U$-modules $M_U/p^nM_U\to M_U/pM_U$,
is surjective for every $n\geq 1$ (here $M_U$ denotes the restriction of the $G$-module $M$ to $U$).
By Kummer theory, the maximal pro-$p$ Galois group of a field $\K$ containing a primitive $p$-th root of 1 is 1-smooth (see Theorem~\ref{thm:1smooth Gal}).
We prove that the pro-$p$ groups defined in Theorem~\ref{thm:main} are not 1-smooth.

After the positive solution of the celebrated Bloch-Kato conjecture given by M.~Rost and V.~Voevodsky (with a ``patch'' of C.~Weibel, see \cites{voev,weibel,weibel2}), one knows that the {\sl $\F_p$-cohomology algebra} of the maximal pro-$p$ Galois group $G_{\K}(p)$ of a field $\K$ containing a primitive $p$-th root of 1
\[
 H^\bullet\left(G_{\K}(p),\F_p\right):=\bigoplus_{n\geq 0} H^n\left(G_{\K}(p),\F_p\right),
\]
with $\F_p$ the finite field of $p$ elements considered as a trivial $G_{\K}(p)$-module, and endowed with the cup-product, is a quadratic algebra: 
i.e., all its elements of positive degree are combinations of products of elements of degree 1, and its defining relations are homogeneous relations of degree 2 (see \S~\ref{ssec:BK}).
From this, it was possible to obtain new obstructions for the realization of a pro-$p$ group as maximal pro-$p$ Galois group (see, e.g., \cites{cem,cq:bk,qw:cyc,SZ:RAAGs}).
Subsequently, it has been conjectured that the algebra $H^\bullet(G_{\K}(p),\F_p)$ has no external operations in addition to its ``internal'' ring structure with respect to the cup-product: namely, all {\sl $n$-fold Massey products} of $H^\bullet(G_{\K}(p),\F_p)$ vanish (see \cite[Conj.~1.1]{MT:masseyall} and \S~\ref{ssec:massey}--\ref{ssec:unipotent}).
This conjecture has been shown to be true for $n=3$ and in other relevant cases (see \cites{EM:massey,GMT:massey4,PJ,HW:massey}).

 We show that from Theorem~\ref{thm:main} one may produce several {\sl one-relator} pro-$p$ groups (i.e., pro-$p$ groups defined by a single relation) whose $\F_p$-cohomology algebra is quadratic and yields the vanishing of 3- and 4-fold Massey products, and which are not absolute Galois groups of fields.

 \begin{thm}\label{prop:intro}
 Let $G=\langle\:x_1,\ldots,x_d\:\mid\:r_0\:\rangle$ be a one-relator pro-$p$ group with $r_0$ as in Theorem~\ref{thm:main}.
 Then the following hold.
 \begin{itemize}
  \item[(i)] The $\F_p$-cohomology algebra $H^\bullet(G,\F_p)$ is quadratic.
  \item[(iia)] If $r_0$ is of the first type, then every 3-fold and 4-fold Massey product vanishes in $H^\bullet(G,\F_p)$, unless $p=q=3$.
  \item[(iib)] If $r_0$ is of the second type, then every 3-fold and 4-fold Massey product vanishes in $H^\bullet(G,\F_p)$, unless $p=q=3$ or $n=2,3$.
  \end{itemize}
 \end{thm}
 
Therefore, one may not employ these cohomological conditions to rule out the pro-$p$ groups defined in Theorem~\ref{prop:intro} as maximal pro-$p$ Galois groups and absolute Galois groups of fields. 
Moreover, most of these one-relator pro-$p$ groups do not belong to the family of pro-$p$ groups which are not absolute Galois groups introduced in \cite[Cor.~1]{BLMS} (see \S~\ref{ssec:BLMS}).
Hence, Theorems~\ref{thm:main}--\ref{prop:intro} provide a big wealth of genuine new examples of pro-$p$ groups that are not absolute Galois groups of fields --- see also Remark~\ref{rem:final}.

\bigskip

{\small
The paper is structured as follows.
In Section~\ref{sec:cohom} we give basic definitions and properties on the cohomology of pro-$p$ groups (in particular, in Subsection~\ref{ssec:kummer} we give the definition of 1-smooth pro-$p$ group).
In Section~\ref{sec:3} we prove Theorem~\ref{thm:main} (cf. Subsections~\ref{ssec:G2}--\ref{ssec:G1}).
In Section~\ref{sec:onerel} we prove Theorem~\ref{prop:intro}--(i), and we show that our one-relator pro-$p$ groups are different to the pro-$p$ groups defined in \cite{BLMS} (cf. Proposition~\ref{prop:BLMS}).
In Section~\ref{sec:massey} we give a brief (and self-contained) tractation of the group-theoretic interpretation of Massey products in $\F_p$-cohomology of pro-$p$ groups (cf. Subsection~\ref{ssec:unipotent}), which we use to prove Theorem~\ref{prop:intro}--(iia)--(iib) (cf. Subsections~\ref{ssec:masseyG2}--\ref{ssec:masseyG1}).}


\section{Pro-$p$ groups and Galois cohomology}\label{sec:cohom}

We work in the world of pro-$p$ groups.
Henceforth, every subgroup of a pro-$p$ group will be tacitly assumed to be closed, and the generators of a subgroup will be intended in the topological sense.

In particular, for a pro-$p$ group $G$ and a positive integer $n$, $G^n$ will denote the closed subgroup of $G$ generated by the $n$-th powers of all elements of $G$.
Moreover, for two elements $g,h\in G$, we set $$h^g=g^{-1}hg,\qquad\text{and}\qquad[h,g]=h^{-1}\cdot h^g,$$
and for two subgroups $H_1,H_2$ of $G$, $[H_1,H_2]$ will denote the closed subgroup of $G$ generated by all commutators $[h,g]$
with $h\in H_1$ and $g\in H_2$.
In particular:
\begin{itemize}
 \item[(a)] $G'$ will denote the commutator subgroup $[G,G]$ of $G$;
 \item[(b)] $\Phi(G)=G^p\cdot G'$ will denote the Frattini subgroup of $G$;
 \item[(c)] $(\gamma_n(G))_{n\geq1}$ will denote the descending central series of $G$, i.e., $\gamma_1(G)=G$, $\gamma_2(G)=G'$ and $\gamma_{n+1}(G)=[\gamma_n(G),G]$ for $n\geq2$.
 \item[(d)] $G_{(3)}$ will denote the subgroup of $G$ defined by
 \[ G_{(3)}= \begin{cases} G^p\cdot\gamma_3(G) & \text{if }p\neq2, \\ G^4\cdot (G')^2\cdot\gamma_3(G) &\text{if }p=2,
  \end{cases} \]
i.e., $G_{(3)}$ is the 3-rd term of the {\sl $p$-Zassenhaus filtration} of $G$ (cf., e.g., \cite[\S~1]{ido:massey}) --- note that $\Phi(G)/G_{(3)}$ is a $p$-elementary abelian pro-$p$ group.
\end{itemize}

Finally, recall that $\Z_p$ denotes the ring of $p$-adic integers
\[\Z_p=\left\{\:a_0+a_1p+a_2p^2+\ldots+a_np^n+\ldots\:\mid\: 0\leq a_i\leq p-1\;\forall\:i\geq0\:\right\},\]
which is a cyclic pro-$p$ group (endowed with the addition), which may be generated by any $p$-adic integer $1+p\lambda$, $\lambda\in\Z_p$.

For further properties of pro-$p$ groups, we refer to \cite[Ch.~1--4]{ddsms}.


\subsection{Cohomology of pro-$p$ groups}
In this subsection, we focus on finitely generated pro-$p$ groups.

Let $G$ be a pro-$p$ group, and let $\F_p$ denote the finite field with $p$-elements.
We consider $\F_p$ also as a trivial $G$-module.
For the 0-th and the 1-st cohomology groups of $G$ with coefficients in $\F_p$ we have isomorphisms of $\F_p$-vector spaces $H^0(G,\F_p)\simeq \F_p$ and
\begin{equation}\label{eq:H1}
 H^1(G,\F_p)=\mathrm{Hom}_{\mathrm{gp}}(G,\F_p)\simeq(G/\Phi(G))^\ast,
\end{equation}
where $\mathrm{Hom}_{\mathrm{gp}}(\:\_\:,\F_p)$ denotes the group of (pro-$p$) group homomorphisms to the additive group of $\F_p$, and $\_^\ast=\Hom_{\F_p}(\:\_\:,\F_p)$ denotes the $\F_p$-dual (cf. \cite[Prop.~3.9.1]{nsw:cohn}).
If $\mathcal{X}=\{x_1,\ldots,x_d\}$ is a minimal set generating $G$, then by \eqref{eq:H1} one has $d=\dim(H^1(G,\F_p))$, and $\mathcal{X}$ yields a dual basis $\mathcal{B}=\{\chi_1,\ldots,\chi_d\}$ of $H^1(G,\F_p)$, i.e., $\chi_i(x_j)=\delta_{ij}$ for every $1\leq i,j\leq d$.

A short exact sequence of pro-$p$ groups $\{1\} \to  N  \to   G \to   G/N \to   \{1\}$
induces an exact sequence of $\F_p$-vector spaces 
\begin{equation}\label{eq:5tes}
  \begin{tikzpicture}[descr/.style={fill=white,inner sep=2pt}]
        \matrix (m) [
            matrix of math nodes,
            row sep=3em,
            column sep=4.5em,
            text height=1.5ex, text depth=0.25ex
        ]
        {  0 & H^1(G/N,\F_p) & H^1(G,\F_p) & H^1(N,\F_p)^G \\
            & H^2(G/N,\F_p) & H^2(G,\F_p) & \\
           };

        \path[overlay,->, font=\scriptsize,>=latex]
        (m-1-1) edge  (m-1-2) 
        (m-1-2) edge node[auto] {$\mathrm{inf}_{G/N,G}^1$} (m-1-3) 
        (m-1-3) edge node[auto] {$\mathrm{res}^1_{G,N}$} (m-1-4)
        (m-1-4) edge[out=355,in=175] node[descr,yshift=0.3ex] {$\mathrm{trg}$} (m-2-2)
        (m-2-2) edge node[auto] {$\mathrm{inf}_{G/N,G}^2$} (m-2-3);
\end{tikzpicture}
\end{equation}
(cf. \cite[Prop.~1.6.7]{nsw:cohn}).

A short exact sequence of pro-$p$ groups
\begin{equation}\label{eq:presentation}
\xymatrix{ \{1\} \ar[r] & R  \ar[r] &  F \ar[r] &  G \ar[r] &  \{1\} },
\end{equation}
with $F$ a free pro-$p$ group, is called a {\sl minimal presentation} if $R\subseteq\Phi(F)$, or, equivalently, if the map $\mathrm{inf}_{G,F}^1$ induced by the epimorphism $F\to G$ is an isomorphism.
Moreover, since $H^2(F,\F_p)=0$ (cf. \cite[Prop.~3.9.5]{nsw:cohn}), the map $\mathrm{trg}$ induces an isomorphism of $\F_p$-vector spaces
\begin{equation}\label{eq:H2}
\mathrm{trg}^{-1}\colon H^2(G,\F_p)\overset{\sim}{\longrightarrow} H^1(R,\F_p)^F\simeq (R/R^p[F,R])^\ast.
\end{equation}
If $\{r_i,i\in\mathcal{I}\}\subseteq R$ is a minimal set generating $R$ as normal subgroup of $F$, and if $\mathcal{X}=\{x_1,\ldots,x_d\}$ is a minimal set generating $G$, we write $G=\langle\: x_1,\ldots,x_d\:\mid\: r_i,i\in\mathcal{I}\:\rangle$, and by \eqref{eq:H2} one has $\dim(H^2(G,\F_p))=|\mathcal{I}|$.
A pro-$p$ group $G$ with minimal presentation \eqref{eq:presentation} is a {\sl one-relator} pro-$p$ group if $R$ is generated as normal subgroup of $F$ by a single element $r$, i.e., if $\dim(H^2(G,\F_p))=1$.

The $\F_p$-cohomology of a pro-$p$ group comes endowed with the cup-product
\[\xymatrix{ H^s(G,\F_p)\times H^t(G,\F_p)\ar[r]^-\cup & H^{s+t}(G,\F_p)},\qquad s,t\geq0\]
which is graded-commutative, i.e., $\beta\cup\alpha=(-1)^{st}\alpha\cup\beta$ for $\alpha\in H^s(G,\F_p),\beta\in H^t(G,\F_p)$ (cf. \cite[Ch.~I, \S~4]{nsw:cohn}).

For finitely generated one-relator pro-$p$ groups one has the following (cf. \cite[Prop.~4.2 and Prop.~4.6]{cq:onerel}).

\begin{prop}\label{prop:onerel}
Let $G=\langle \:x_1,\ldots,x_d\:\mid\: r\:\rangle$ be a finitely generated one-relator pro-$p$ group.
If  
\begin{equation}\label{eq:onerel prop}
  r\equiv [x_1,x_2]\cdot[x_3,x_4]\cdots[x_{n-1},x_n]\mod G_{(3)},
\end{equation}
for some $n$ even, $2\leq n\leq d$, then $H^\bullet(G,\F_p)$ is quadratic.
In particular, if $\mathcal{B}=\{\chi_1,\ldots,\chi_d\}$ is a basis of $H^1(G,\F_p)$ dual to $\mathcal{X}=\{x_1,\ldots,x_d\}$, then one has 
\[\begin{split}
   & \chi_1\cup\chi_2=\chi_3\cup\chi_4=\ldots=\chi_{n-1}\cup\chi_n,\\
   & \chi_i\cup\chi_j=0\qquad \text{for }1\leq i\leq j\leq d, (i,j)\neq(1,2),\ldots(n-1,n),
  \end{split}
\]
and $\chi_1\cup\chi_2$ generates $H^1(G,\F_p)$.
Moreover, $H^k(G,\F_p)=0$ for $k\geq3$.
\end{prop}

One-relator pro-$p$ groups satisfying \eqref{eq:onerel prop} generalize {\sl Demushkin groups}, which were studied by S.P.~Demu\v{s}kin, J-P.~Serre and J.P.~Labute (cf. e.g., \cite[Ch.~III, pp.~231--244]{nsw:cohn}) --- in fact, Proposition~\ref{prop:onerel} is based on the description of the cup-product for Demushkin groups.

For further facts on cohomology of pro-$p$ groups we direct the reader to \cite[Ch.~I, \S~3--4]{serre:gc} and to \cite[Ch.~III, \S~9]{nsw:cohn}.

\subsection{Quadratic cohomology and Rost-Voevodsky Theorem}\label{ssec:BK}

Let $G$ be a pro-$p$ group.
The $\F_p$-cohomology $$H^\bullet(G,\F_p)=\bigoplus_{n\geq0} H^n(G,\F_p),$$
endowed with the cup-product, is a non-negatively graded algebra over $\F_p$.
A spcecial class of non-negatively graded $\F_p$-algebras is given by the following (cf. \cite{poliposi:book}).

\begin{defin}\label{defin:quad}\rm
Let $A_\bullet=\bigoplus_{n\geq0}A_n$ be a (non-negatively) graded $\F_p$-algebra, with $A_0=\F_p$.
Then $A_\bullet$ is said to be quadratic if one has an isomorphism of graded algebras
\[
 A_\bullet\simeq \frac{T_\bullet(A_1)}{(\Omega)},
\]
where $T_\bullet(A_1)=\bigoplus_{n\geq0}A_1^{\otimes n}$ denotes the $\F_p$-tensor algebra generated by $A_1$, and $(\Omega)$ is the two-sided ideal of $T_\bullet(A_1)$ generated by $\Omega$, with $\Omega\subseteq A_1\otimes A_1$.
\end{defin}

\begin{exam}\label{ex:quad}\rm
Let $V$ be a finite-dimensional $\F_p$-vector space. Then the symmetric algebra $S_\bullet(V)$ and the exterior algebra $\bigwedge_\bullet(V)$ are quadratic algebras.
\end{exam}

A consequence of the Rost-Voevodsky Theorem is that the $\F_p$-cohomology algebra of the maximal pro-$p$ Galois group of a field containing a primitive $p$-th root of 1 is quadratic (cf., e.g., \cite[\S~24.3]{ido:book} and \cite[\S~2]{cq:bk}).

\begin{thm}\label{thm:BK}
 Let $\K$ be a field containing a primitive $p$-th root of 1.
 Then the $\F_p$-cohomology algebra $H^\bullet(G_{\K}(p),\F_p)$ is a quadratic algebra.
\end{thm}

A pro-$p$ group $G$ such that the $\F_p$-cohomology algebra $H^\bullet(H,\F_p)$ is quadratic for every subgroup $H\subseteq G$ is called a {\sl Bloch-Kato pro-$p$ group} (cf. \cite{cq:bk}).
By Theorem~\ref{thm:BK}, the maximal pro-$p$ Galois group of a field $\K$ containing a primitive $p$-th root of 1 is Bloch-Kato, as every subgroup $H\subseteq G_{\K}(p)$ is the maximal pro-$p$ Galois group of an extension of $\K$.

\begin{rem}\label{rem:n rel BK}\rm
 If $p\neq2$ and $G$ is a finitely generated Bloch-Kato pro-$p$ group, then
 \[
  \dim(H^2(G,\F_p))\leq\binom{\dim(H^1(G,\F_p))}{2}
 \]
(cf. \cite[\S~4.1]{cq:bk}).
Thus, if $G$ is as in Theorem~\ref{thm:main} and $\nu>\binom{d}{2}$, then $G$ is not Bloch-Kato.
\end{rem}

By results of Ware (cf. \cite{ware}) and Engler, Koenigsmann and Nogueira (cf. \cites{EK,EN}), if $\K$ is a field containing a primitive $p$-th root of 1, then the maximal pro-$p$ Galois group $G_{\K}(p)$ has the following properties: (i) if $G_{\K}(p)$ is not meta-abelian (i.e., $G_{\K}(p)'$ is abelian) then $G_{\K}(p)$ contains a free non-abelian subgroup; and (ii) $G_{\K}(p)$ contains a unique maximal abelian normal subgroup $A$, and $G_{\K}(p)\simeq A\rtimes (G_{\K}(p)/A)$.
The following result extends these properties of maximal pro-$p$ Galois groups to Bloch-Kato pro-$p$ groups (cf. \cite[Thm.~B]{cq:bk} and \cite[Thm.~1.2]{qw:cyc}).

\begin{prop}\label{prop:BK split}
 Let $G$ be a Bloch-Kato pro-$p$ group.
\begin{itemize}
 \item[(i)] If $G$ is not meta-abelian, then $G$ contains a free non-abelian subgroup.
 \item[(ii)] If $G$ is also 1-smooth, then $G$ contains a unique maximal abelian subgroup $A$, and $G\simeq A\rtimes G/A$.
\end{itemize}
\end{prop}


\subsection{Kummerian pro-$p$ pairs}\label{ssec:kummer}

Let $1+p\Z_p=\{1+p\lambda\mid \lambda\in\Z_p\}$ denote the pro-$p$ Sylow subgroup of the group of units of the ring of $p$-adic integers $\Z_p$.
A pair $\calG=(G,\theta)$ consisting of a pro-$p$ group and a continuous homomorphism $$\theta\colon G\longrightarrow1+p\Z_p$$ is called a {\sl cyclotomic pro-$p$ pair}, and the morphism $\theta$ is called an {\sl orientation} of $G$ (cf. \cite[\S~3]{ido:small} and \cite{qw:cyc}).
A cyclotomic pro-$p$ pair $\calG=(G,\theta)$ is said to be {\sl torsion-free} if $p$ is odd, or if $p=2$ and $\Img(\theta)\subseteq 1+4\Z_2$. (Note that the pro-$p$ group $G$ might have non-trivial torsion even though $\calG=(G,\theta)$ is a torsion-free cyclotomic pro-$p$ pair.)
Given a cyclotomic pro-$p$ pair $\calG=(G,\theta)$ one has the following constructions:
\begin{itemize}
 \item[(a)] if $H$ is a subgroup of $G$, we set $\calG\vert_H:=(H,\theta\vert_H)$;
 \item[(b)] if $N$ is a normal subgroup of $G$ contained in $\Ker(\theta)$, then $\theta$ induces an orientation $\bar \theta\colon G/N\to1+p\Z_p$, and we set $\calG/N:=(G/N,\bar\theta)$;
\end{itemize}

Given a cyclotomic pro-$p$ pair $\calG=(G,\theta)$, the pro-$p$ group $G$ comes endowed with a distinguished subgroup $K(\calG)$, introduced in \cite[\S~3]{eq:kummer} and defined by
\begin{equation}\label{eq:K subgroup}
   K(\calG)=\left\langle\:\left. h^{-\theta(g)}\cdot h^{g^{-1}}\:\right|\:g\in G,h\in\Ker(\theta)\:\right\rangle.
\end{equation}
The subgroup $K(\calG)$ is normal in $G$, and it is contained in $\Ker(\theta)$.
Moreover, one has $ K(\calG)\supseteq\Ker(\theta)'$ and $K(\calG)\subseteq\Phi(G)$.
Thus, the quotient $\Ker(\theta)/K(\calG)$ is abelian.

\begin{defin}\label{defin:kummer}\rm
Given a cyclotomic pro-$p$ pair $\calG=(G,\theta)$, let $\Z_p(\theta)$ denote the continuous $G$-module of rank 1 induced by $\theta$, i.e., $\Z_p(\theta)\simeq\Z_p$ as abelian pro-$p$ groups, and $g.\lambda=\theta(g)\cdot\lambda$ for every $\lambda\in\Z_p(\theta)$.
The pair $\calG$ is said to be {\sl Kummerian} if for every $n\geq1$ the map
\begin{equation}\label{eq:epimorphism}
  H^1(G,\Z_p(\theta)/p^n\Z_p(\theta))\longrightarrow H^1(G,\F_p),
\end{equation}
induced by the epimorphism of $G$-modules $\Z_p(\theta)/p^n\Z_p(\theta)\to\Z_p(\theta)/p\Z_p(\theta)\simeq\F_p$, is surjective.
Moreover, $\calG$ is {\sl 1-smooth} if $\calG\vert_H$ is Kummerian for every subgroup $H\subseteq G$.
\end{defin}

We say that a pro-$p$ group $G$ may complete into a Kummerian, or 1-smooth, pro-$p$ pair if there exists an orientation $\theta\colon G\to1+p\Z_p$ such that the pair $(G,\theta)$ is Kummerian, respectively 1-smooth.

Given a field $\K$ ontaining a primitive $p$-th root of 1, let $\theta_{\K}\colon G_{\K}(p)\to1+p\Z_p$ denote the {\sl pro-$p$ cyclotomic character}, i.e., if $\zeta$ is a primitive $p^k$-th root of 1, with $k\geq1$, then $\sigma(\zeta)=\zeta^{\theta_{\K}(\sigma)}$ for all $\sigma\in G_{\K}(p)$ (cf., e.g., \cite[\S~4]{eq:kummer}). 
The continuous $G_{\K}(p)$-module $\Z_p(\theta_{\K})$ is called the {\sl 1st Tate twist module}, and usually it is denoted by $\Z_p(1)$ (cf. \cite[Def.~7.3.6]{nsw:cohn}).
By Kummer theory, one has the following (cf. \cite[Thm.~4.2]{eq:kummer} and \cite[Prop.~2.3]{qw:cyc}).

\begin{thm}\label{thm:1smooth Gal}
 Let $\K$ be a field containing a primitive $p$-th root of 1.
 Then the cyclotomic pro-$p$ pair $\calG_{\K}=(G_{\K}(p),\theta_{\K})$ is 1-smooth.
\end{thm}

Kummerian pro-$p$ pairs and 1-smooth pro-$p$ pairs were introduced in \cite{eq:kummer} and in \cite[\S~14]{dcf:lift} respectively.
In \cite{qw:cyc}, if $\calG=(G,\theta)$ is a 1-smooth pro-$p$ pair, the orientation $\theta$ is said to be {\sl 1-cyclotomic}.

\begin{rem}\rm
In \cite[\S~14.1]{dcf:lift}, a pro-$p$ pair is defined to be 1-smooth if the maps \eqref{eq:epimorphism} are surjective for every {\sl open} subgroup of $G$, yet by a limit argument this implies also that the maps \eqref{eq:epimorphism} are surjective also for every {\sl closed} subgroup of $G$ (cf. \cite[Cor.~3.2]{qw:cyc}). 
\end{rem}

Henceforth we will restrict our attention to torsion-free cyclotomic pro-$p$ pairs whose pro-$p$ group is finitely generated.
One has the following group-theoretic characterization of Kummerianity (cf. \cite[Thm.~5.6 and Thm.~7.1]{eq:kummer} and \cite[Thm.~5.9 and Cor.~5.10]{cq:1smooth}).

\begin{thm}\label{thm:kummer}
Let $\calG=(G,\theta)$ be a torsion-free cyclotomic pro-$p$ pair $\calG=(G,\theta)$, with $G$ finitely generated.
The following are equivalent.
\begin{itemize}
 \item[(i)] The pair $\calG$ is Kummerian.
 \item[(ii)] The quotient $\Ker(\theta)/K(\calG)$ is a free abelian pro-$p$ group.
 \item[(iii)] The pair $\calG/N$ is Kummerian for every normal subgroup $N$ of $G$ contained in $K(\calG)$.
\end{itemize}
\end{thm}

\begin{rem}\label{rem:abs torfree}\rm
Let $\calG=(G,\theta)$ be a cyclotomic pro-$p$ pair with $\theta\equiv\mathbf{1}$, i.e., $\theta$ is constantly equal to 1.
Since $K(\calG)=G'$ in this case, by Theorem~\ref{thm:kummer} the pair $\calG$ is Kummerian if and only if the quotient $G/G'$ is torsion-free.
Moreover, $\calG$ is 1-smooth if and only if $H/H'$ is torsion-free for every subgroup $H\subseteq G$.
Pro-$p$ groups with such a property are called {\sl absolutely torsion-free}, and they were introduced and studied by W\"urfel in \cite{wurf}.
\end{rem}

\begin{examples}\label{ex:1-smooth}\rm
\begin{itemize}
\item[(a)] A cyclotomic pro-$p$ pair $(G,\theta)$ with $G$ a free pro-$p$ group is 1-smooth for any orientation $\theta\colon G\to1+p\Z_p$ (cf. \cite[\S~2.2]{qw:cyc}).
\item[(b)] A cyclotomic pro-$p$ pair $(G,\theta)$ with $G$ an infinite Demushkin pro-$p$ group is 1-smooth if and only if $\theta\colon G\to1+p\Z_p$ is the canonical orientation of $G$, defined as in \cite[Thm.~4]{labute:demushkin} (cf. \cite[Thm.~7.6]{eq:kummer}).
\item[(c)] The only 1-smooth pro-$p$ pair $(G,\theta)$ with $G$ a finite $p$-group is the cyclic group of order 2 $G\simeq\Z/2$, endowed with the only non-trivial orientation $\theta\colon G\twoheadrightarrow\{\pm1\}\subseteq 1+2\Z_2$ (cf. \cite[Ex.~3.5]{eq:kummer}).
\end{itemize}
\end{examples}

\begin{rem}\label{rem:torfree}\rm
By Example~\ref{ex:1-smooth}--(c), if $\calG=(G,\theta)$ is a torsion-free 1-smooth pro-$p$ pair (i.e., $\Img(\theta)$ is torsion-free), then $G$ is a torsion-free pro-$p$ group.
\end{rem}

Kummerianity is inherited by certain quotients of pro-$p$ pairs (cf. \cite[Thm.~5.6]{cq:1smooth}).

\begin{prop}\label{thm:kummer quot}
 Let $\calG=(G,\theta)$ a torsion-free Kummerian pro-$p$ pair, an let $N$ be a normal subgroup of $G$ satisfying:
\begin{itemize}
 \item[(i)] $G/N$ is a finitely generated pro-$p$ group;
 \item[(ii)] $N\subseteq\Ker(\theta)$;
 \item[(iii)] the map $\res_{G,N}^1\colon H^1(G,\F_p)\to H^1(N,\F_p)^G$ is surjective.
\end{itemize}
Then also the pro-$p$ pair $\calG/N$ is Kummerian.
\end{prop}

\begin{rem} \label{rem:Kummer quot}\rm
By \eqref{eq:5tes} condition~(iii) of Proposition~\ref{thm:kummer quot} amounts to requiring that the inclusion $N\subseteq G$ induces a monomorphism of $p$-elementary abelian pro-$p$ groups $N/N^p[G,N]\hookrightarrow G/\Phi(G)$.
\end{rem}

Let $\calG=(G,\theta)$ be a cyclotomic pro-$p$ pair.
A {\sl continuous 1-cocycle} from $G$ to the $G$-module $\Z_p(\theta)$ is a continuous map $c\colon G\to\Z_p(\theta)$ satisfying
\begin{equation}\label{eq:1cocyc}
 c(xy)=c(x)+\theta(x)\cdot c(y)\qquad\text{for every }x,y\in G
\end{equation}
(cf., e.g., \cite[Ch.~I, \S~2]{nsw:cohn}).
One may employ continuous 1-cocycles to check whether a given cyclotomic pro-$p$ pair $\calG$ is Kummerian
(cf. \cite[Prop.~6]{labute:demushkin}).

\begin{prop}\label{prop:kummer 1cocyc}
Let $\calG=(\calG,\theta)$ be a torsion-free cyclotomic pro-$p$ pair with $G$ finitely generated, and let $\mathcal{X}=\{x_1,\ldots,x_d\}$ be a minimal generating set of $G$.
Then $\calG$ is Kummerian if, and only if, for any sequence $(\lambda_1,\ldots,\lambda_d)$ with $\lambda_i\in\Z_p$ for every $i=1,\ldots,d$, there exists a continuous 1-cocycle $c\colon G\to\Z_p(\theta)$ such that $c(x_i)=\lambda_i$ for every $i=1,\ldots,d$.
\end{prop}

\begin{examples}\label{ex:nokummer 1cocyc}\rm
\begin{itemize}
 \item[(a)] For $p\neq2$ let $G$ be the pro-$p$ group with minimal presentation 
$$ G=\langle\: x,y,z\:\mid\:  [x,y]=z^p\:\rangle,$$
and let $\theta\colon G\to1+p\Z_p$ be an arbitrary orientation.
Then there are no continuous 1-cocycles $c\colon G\to\Z_p(\theta)$ satisfying $c(x)=c(y)=0$ and $c(z)=1$, and thus the pro-$p$ pair $(G,\theta)$ is not Kummerian for any orientation $\theta$ (cf. \cite[Thm.~8.1]{eq:kummer}).
\item[(b)] Let 
$$H=\left\{\:\left(\begin{array}{ccc} 1 & a & c \\ 0 & 1 & b \\ 0 & 0 &1 \end{array}\right)\:\mid\: a,b,c\in\Z_p\:\right\}$$
 be the Heisenberg pro-$p$ group.
By Remark~\ref{rem:abs torfree}, the pair $(H,\mathbf{1})$ is Kummerian, as $H/H'\simeq\Z_p^2$, but $H$ is not absolutely torsion-free.
In particular, $H$ contains an open subgroup $U$ which is isomorphic to the pro-$p$ group $G$ as in (a), and therefore $H$ cannot complete into a 1-smooth pro-$p$ pair (cf. \cite[Ex.~6.6]{cq:1smooth}).
\end{itemize}
\end{examples}



\section{Pro-$p$ groups that are not absolute galois groups}\label{sec:3}

Let $x,y,z$ be elements of a (pro-$p$) group $G$, and $n\geq1$.
Recall that by commutator calculus one has the equalities
\[ \begin{split}
   [xy,z]=[x,z]^y\cdot [y,z]\quad&\text{and}\quad[x,yz]=[x,z]\cdot[x,y]^z, \\
   [x^n,y]=[x,y]^{x^{n-1}}\cdots[x,y]^x\cdot[x,y]\quad&\text{and}\quad[x,y^n]=[x,y]\cdot[x,y]^{y}\cdots[x,y]^{y^{n-1}}.
 \end{split}\]
\begin{lem}\label{lem:comm}
 Let $x,y$ be elements of a (pro-$p$) group $G$.
 For $n\geq1$ one has the following:
 \begin{equation}\label{eq:comm}
  [y,x^n]\equiv \prod_{i=1}^n[y,_ix]^{\binom{n}{i}}\mod G'',
 \end{equation}
where $G''$ denotes the commutator subgroup of $G'$.
\end{lem}

\begin{proof}
We proceed by induction on $n$.
If $n=2$, then 
\[[y,x^2]=[y,x]\cdot[y,x]^x=[y,x]^2[y,_2x].\]
Now assume that \eqref{eq:comm} holds for $n-1$. 
Then 
\[\begin{split}
   [y,x^n] & = [y,x][y,x^{n-1}]^x \\ &= [y,x][y,x^{n-1}][[y,x^{n-1}],x] \\
   &\equiv [y,x]\cdot \prod_{i=1}^{n-1}[y,_ix]^{\binom{n-1}{i}}\cdot\left[\prod_{i=1}^{n-1}[y,_{i}x]^{\binom{n-1}{i}},x\right]\mod G''\\
   &\equiv \prod_{i=1}^{n}\left([y,_ix]^{\binom{n-1}{i}}\cdot[y,_ix]^{\binom{n-1}{i-1}}\right)\mod G''\\
  \end{split}\]
(here we set implicitly $\binom{n-1}{n}=0$),
  and the well-known equality $\binom{n-1}{i}+\binom{n-1}{i-1}=\binom{n}{i}$ yields \eqref{eq:comm}.
\end{proof}

Also, we recall the following (cf. \cite[Lemma~6.1]{eq:kummer}).

\begin{lem}\label{lem:1cocyc}
 Let $\calG=(G,\theta)$ be a torsion-free cyclotomic pro-$p$ pair, and let $c\colon G\to\Z_p(\theta)$ be a continuous 1-cocycle.
 Then for every $x,y\in G$ and $\lambda\in \Z_p$, one has 
 \[  \begin{split}
   c(x^\lambda) &=\begin{cases} \lambda\cdot c(x) & \text{if }\theta(x)=1, \\ 
   \dfrac{\theta(x)^\lambda-1}{\theta(x)-1}\cdot c(x) & \text{if }\theta(x)\not=1;      \end{cases}\\
   c([x,y])&=\theta(xy)^{-1}((1-\theta(y))c(x)-(1-\theta(x))c(y)).
  \end{split} \]
In particular, $c(1)=0$, and for $n\geq1$ one has
\[
c([y,_nx])=\left(\theta(x)^{-1}-1\right)^{n-1}c([y,x]). 
\]
  \end{lem}
  
Finally, we underline the following straightforward Galois-theoretic observation.
  
\begin{rem}\label{rem:GKquot}\rm
Let $\K$ be a field. If $p=2$ then clearly $\K$ contains the primitive 2-nd root of 1, i.e., $-1$; on the other hand, if $p\neq 2$ and $\K$ does not contain a primitive $p$-th root of 1, say $\zeta_p\in\bar\K_s$, then $[\K(\zeta_p):\K]=p-1$, and the absolute Galois group $\Gal(\bar \K_s/\K)$ is not a pro-$p$ group.
Therefore, if $\Gal(\bar \K_s/\K)$ is a pro-$p$ group then $\zeta_p\in\K$.

This yields the following: if a pro-$p$ group does not occur as maximal pro-$p$ Galois group of a field containing a primitive $p$-th root of 1, then it does not occur as absolute Galois group of any field.
\end{rem}


\subsection{The first type of $r_0$}\label{ssec:G2}

In this subsection, $G$ denotes a pro-$p$ group with minimal presentation $ G=\langle\: x_1,\ldots,x_d\:\mid\:r_0,r_1\ldots,r_\nu\:\rangle$ as in Theorem~\ref{thm:main}, with $r_0$ of the first type, i.e., $d$ is odd and
\[
 r_0=x_1^q\cdot[x_1,_nx_2]\cdot [x_2,x_3]\cdots[x_{d-1},x_d]\cdot s,
\]
with $s\in S''$, where $S$ is the subgroup of $G$ generated by $x_3,\ldots,x_d$; while $r_1,\ldots,r_\nu\in S'$.
Our goal is to show that for any orientation $\theta\colon G\to1+p\Z_p$, the pro-$p$ pair $\calG=(G,\theta)$ is not 1-smooth.

First, we establish what a suitable orientation $\theta\colon G\to1+p\Z_p$ should be like.

\begin{prop}\label{prop:theta 2}
 Let $\theta\colon G\to1+p\Z_p$ be an orientation of $G$ such that the cyclotomic pro-$p$ pair $\calG=(G,\theta)$ is torsion-free. 
 If $\calG$ is Kummerian then $q=0$ and $\theta(x_i)=1$ for $i=2,\ldots,d$.
\end{prop}

\begin{proof}
First, note that $\theta\vert_{S'}\equiv\mathbf{1}$, and thus by Lemma~\ref{lem:1cocyc} one has $c\vert_{S''}\equiv\mathbf{0}$ for every continuous 1-cocycle $c\colon G\to\Z_p(\theta)$.
 From
 \[      1=\theta(1) = \theta\left(x_1^ q[x_1,_nx_2][x_3,x_4]\cdots[x_{d-1},x_d]\cdot s\right) =\theta(x_1)^ q
   \]
we deduce that $\theta(x_1)^ q=1$.
Since $\Img(\theta)$ is torsion-free, if $ q\neq0$ one has $\theta(x_1)=1$, while one might have $\theta(x_1)\neq1$ only if $ q=0$. 

For every $i=1,\ldots,d$, let $c_i\colon G\to\Z_p(\theta)$ be the continuous 1-cocycle such that $c_i(x_j)=\delta_{ij}$, with $j\in\{1,\ldots,d\}$ (such a continuous 1-cocycle always exists by Proposition~\ref{prop:kummer 1cocyc}).
Then 
\begin{equation}\label{eq:1cocyc G2}
 \begin{split}
 0 &=c_i\left(x_1^ q[x_1,_nx_2][x_2,x_3]\cdots[x_{d-1},x_d]\cdot s\right)\\
 &= c_i(x_1^ q)+\theta(x_1)^ q\cdot c_i\left([x_1,_nx_2][x_3,x_4]\cdots[x_{d-1},x_d]\cdot s\right) \\
 &= c_i(x_1^ q)+c_i([x_1,_nx_2])+c_i([x_3,x_4])+\ldots+c_i([x_{d-1},x_d]).
 \end{split} 
\end{equation}
For $i\geq4$, $i$ even, from \eqref{eq:1cocyc}, Lemma~\ref{lem:1cocyc}, and \eqref{eq:1cocyc G2}, we get
\[
0= 0+\ldots+c_i([x_i,x_{i+1}])+\ldots+0=\theta(x_ix_{i+1})^{-1}(1-\theta(x_{i+1}))\cdot1,
 \]
and hence $\theta(x_{i+1})=1$.
Analogously, for $i\geq3$, $i$ odd, one gets 
\[0= 0+\ldots+c_i([x_{i-1},x_{i}])+\ldots+0=\theta(x_{i-1}x_i)^{-1}(-1+\theta(x_{i-1}))\cdot1,\]
and hence $\theta(x_{i-1})=1$.
For $i=2$, from \eqref{eq:1cocyc}, Lemma~\ref{lem:1cocyc}, and \eqref{eq:1cocyc G2}, we get
\[\begin{split}
    0 &= 0+c_2([x_1,_nx_2])+c_2([x_2,x_3])+\ldots+0 \\
    &= -\theta(x_1)(\theta(x_2)^{-1}-1)^{n}+\theta(x_2x_3)^{-1}(1-\theta(x_3)) \\&= -\theta(x_1)(\theta(x_2)^{-1}-1)^{n}+0,
  \end{split}\]
as $\theta(x_3)=1$, and hence $\theta(x_2)=1$.
Finally, for $i=1$ from \eqref{eq:1cocyc}, Lemma~\ref{lem:1cocyc} and \eqref{eq:1cocyc G2} we get
\[ 0=c_1(x_1^ q)+c_1([x_1,_nx_2])+0=\begin{cases}
                                           q & \text{if }\theta(x_1)=1, \\
                                          \frac{\theta(x_1)^q-1}{\theta(x_1)-1} & \text{if }\theta(x_1)\not=1.
                                         \end{cases}\]
Therefore, $q=0$, and one might have $x_1\notin\Ker(\theta)$.
\end{proof}

Note that if $\calG=(G,\theta)$ is a Kummerian pro-$p$ pair, then $S\subseteq \Ker(\theta)$, and thus Lemma~\ref{lem:1cocyc} imples that $c\vert_{S'}\equiv\mathbf{0}$ for every continuous 1-cocycle $c\colon G\to\Z_p(\theta)$.
From now on we assume that $q=0$, so that $G$ may be completed into a Kummerian pro-$p$ pair $\calG=(G,\theta)$.

Let $\phi\colon G\to \Z/p\Z$ be the homomorphism of pro-$p$ groups defined by $x_i\mapsto 0$ for $i=1,3,4,\ldots,d$ and $x_2\mapsto 1$, and set $U=\Ker(\phi)$.

\begin{lem}\label{lem basis U2}
 The set
 \[
\mathcal{X}_U=\left\{\begin{array}{c} t:=x_2^p, \\ 
x_1,[x_1,x_2],\ldots,[x_1,_{l-1}x_2], \\x_3,[x_3,x_2],\ldots,[x_3,_{l'-1}x_2],\\
x_i,[x_i,x_2],\ldots,[x_i,_{p-1}x_2],i=4,\ldots,d
\end{array}\right\} , 
 \]
where $l=\min\{n,p\}$, and $l'=p-n$ if $n< p$, $l'=1$ if $n\geq p$, is a minimal generating set of $U$.
\end{lem}

\begin{proof}
Set $\bar U=U/U'$, and for $g\in U$ let $\bar g=g\cdot U'$ denote its image in $\bar U$.
 In particular, we set $v_{i,0}=\bar x_i$ and $v_{i,k}=\overline{[x_i,_kx_2]}\in \bar U$, for $i=1,\ldots,d$ and $k\geq1$.
  
  The group $U$ is generated as normal subgroup of $G$ by $t$ and by $x_1,x_3,x_4,\ldots,x_d$.
 Since $G'\subseteq U$, one has $G''\subseteq  U'$.
 Also, $S'\subseteq U'$, as $S\subseteq U$.
 By Lemma~\ref{lem:comm}, for every $i=1,3,\ldots, d$ one has $$[x_i,t]\equiv \left(\prod_{k=1}^{p-1}[x_i,_kx_2]^{\binom{p}{k}}
 \right)\cdot[x_i,_px_2]\mod U',$$
 and since $[x_i,t]\in U'$, it follows that for every $k\geq p$
  \begin{equation}\label{eq:commx1x2p G2}
 v_{i,k}=v_{i,1}^{\lambda_1}\cdot v_{i,2}^{\lambda_2}\cdots v_{i,p-2}^{\lambda_{p-2}}\cdot v_{i,p-1}^{\lambda_{p-1}}
 \end{equation}
 for some $\lambda_1,\ldots,\lambda_{p-1}\in p\Z_p$.
 Hence, $\bar U$ is generated as abelian pro-$p$ group by the set $\{\bar t, v_{i,k}\}_{i,k}$, with $i=1,3,\ldots,d$ and $0\leq k\leq p-1$.

In $U$, from the defining relation $r_0=1$ one obtains $p$ relations
\begin{equation}\label{eq:p relations 1st}
 \begin{split}
r_0=[x_1,_nx_2][x_2,x_3]\cdots[x_{d-1},x_d] &=1,\\
 [r_0,x_2]=\left[ [x_1,_nx_2][x_2,x_3]\cdots[x_{d-1},x_d],x_2\right] &=1,\\
 &\vdots\\  [r_0,_{p-1}x_2]=\left[ [x_1,_nx_2][x_2,x_3]\cdots[x_{d-1},x_d],_{p-1}x_2\right] &=1.
 \end{split}
\end{equation}
From the relations \eqref{eq:p relations 1st} one deduces 
\begin{equation}\label{eq:rel U2}
 \left[[x_1,_nx_2][x_2,x_3],_{k-1}x_2\right]\equiv [x_1,_{n+k-1}x_2]\cdot[x_3,_{k}x_2]^{-1}\equiv 1\mod U'
\end{equation}
for every $1\leq k\leq p$.
Note that one obtains the abelian pro-$p$ group $\bar U$ as the quotient of the free $\Z_p$-module generated by $\{\bar t,v_{i,k}\}$, with $i=1,3,\ldots,d$ and $0\leq k\leq p-1$, over the $p$ relations \eqref{eq:rel U2}, as the relations $r_1,\ldots,r_\nu$ of $G$ lie in $S'\subseteq U'$.
Moreover, one has $\bar U/\bar U^p=\bar U/\Phi(\bar U)\simeq U/\Phi(U)$.

Suppose first that $n\geq p$.
By applying \eqref{eq:commx1x2p G2} to the $p$ relations \eqref{eq:rel U2}, one obtains $p$ relations in $\bar U$
\begin{equation}\label{eq:rel Ubar 1}
  v_{3,k}=v_{1,n+(k-1)}=\prod_{j=1}^{p-1}v_{1,j}^{\lambda_j},\qquad \text{for }k=1,\ldots,p,
\end{equation}
for some $\lambda_1,\ldots,\lambda_{p-1}\in p\Z_p$ depending on $k$.
Consider the quotient $\bar U/\bar U^p$.
The $p$ relations \eqref{eq:rel Ubar 1} translate into 
\[
 v_{3,k}\equiv 1\mod \bar U^p\qquad\text{for }k=1,\ldots,p.
\]
Therefore, the cosets $\{\bar t\bar U^p,v_{3,0}\bar U^p,v_{i,k}\bar U^p\}_{i,k}$, with $i=1,4,\ldots,d$ and $0\leq k<p$, generate minimally the $p$-elementary abelian group $\bar U/\bar U^p\simeq U/\Phi(U)$, and thus $\mathcal{X}_U$ generates minimally $U$.

Now suppose that $2<n<p$.
By applying \eqref{eq:commx1x2p G2} to the $p$ relations \eqref{eq:rel U2}, one obtains the following $p$ relations in $\bar U$:
\begin{equation}
 \begin{split}\label{eq:rel Ubar 2}
 v_{1,n+(k-1)} = v_{3,k}, &\qquad\text{for } k=1,\ldots, p-n, \\
 v_{3,k} = v_{1,n+(k-1)}=\prod_{j=1}^{p-1}v_{1,j}^{\lambda_j}, &\qquad \text{for } k= p-n+1,\ldots,p,
\end{split}
\end{equation}
for some $\lambda_1,\ldots,\lambda_{p-1}\in p\Z_p$ depending on $k$.
Consider the quotient $\bar U/\bar U^p$.
The $p$ relations \eqref{eq:rel Ubar 2} translate into 
\[\begin{split}
   v_{1,n+(k-1)} \equiv v_{3,k}\mod \bar U^p, &\qquad\text{for } k=1,\ldots, p-n, \\
 v_{3,k} \equiv 1\mod \bar U^p,&\qquad \text{for } k= p-n+1,\ldots,p.
  \end{split} \]
Therefore, the cosets $\{\bar t\bar U^p, v_{1,k}\bar U^p, v_{3,k'}\bar U^p, v_{i,k''}\bar U^p\}_{i,k,k',k''}$, with $i=4,\ldots,d$ and $0\leq k\leq n-1$, $0\leq k'\leq p-n$, $0\leq k''\leq p-1$, generate minimally the $p$-elementary abelian group $\bar U/\bar U^p\simeq U/\Phi(U)$, and thus $\mathcal{X}_U$ generates minimally $U$.
\end{proof}

\begin{prop}\label{prop:G2 UN}
Suppose that $\calG=(G,\theta)$ is a Kummerian torsion-free pro-$p$ pair, and let $N$ be the normal subgroup of $U$ generated (as normal subgroup) by the elements $[x_i,_kx_2]$, with $i=3,\ldots,d$ and $k\geq0$.
Then one has the following:
\begin{itemize}
 \item[(i)] $N\subseteq\Ker(\theta\vert_U)$;
 \item[(ii)] the map $\res_{U,N}^1\colon H^1(U,\F_p)\to H^1(N,\F_p)^U$ induced by the inclusion $N\subseteq U$ is surjective;
 \item[(iii)] the pro-$p$ pairs $\calG\vert_U/N=(U/N,\bar\theta\vert_U)$ and $\calG\vert_U=(U,\theta\vert_U)$ are not Kummerian.
\end{itemize}
\end{prop}

\begin{proof}
 Since $x_i\in\Ker(\theta)$ for every $i=3,\ldots,d$, one has that $N\subseteq \Ker(\theta\vert_U)$.
 Moreover, $S\subseteq N$.
 
Let $\mathcal{B}_U=\{\chi_t,\chi_{1,k},\chi_{3,k'},\chi_{i,k''}\}_{i,k,k',k''}$ denote a basis of $H^1(U,\F_p)$ dual to $\mathcal{X}_U$, with $i=4,\ldots,d$ and $0\leq k\leq l-1$, $0\leq k'\leq l'-1$ and $0\leq k\leq p-1$, where $l,l'$ are as in Lemma~\ref{lem basis U2}.
By duality \eqref{eq:H1}, the image of the map $\mathrm{inf}_{U/N,U}^1\colon H^1(U/N,\F_p)\to H^1(U,\F_p)$ --- which is injective by \eqref{eq:5tes} ---
is the subspace of $H^1(U,\F_p)$ generated by the set $\{\chi_t,\chi_{1,k}\mid 0\leq k\leq l-1\}$.
On the other hand, since the elements $[x_i,_kx_2]$ generate $N$ as a normal subgroup of $U$, one has
\begin{equation}\label{eq:res1 dimension}
\begin{split}
 \dim(H^1(N,\F_p)^U) &=\dim(N/N^p[U,N]) \\ &=\dim(H^1(U,\F_p))-(1+l) \\ &=\dim(H^1(U,\F_p))-\dim(\Img(\mathrm{inf}_{U/N,U}^1)),
\end{split} \end{equation}
and therefore the exactness of \eqref{eq:5tes} implies that the image of the map $\res_{U,N}^1$ is the whole $H^1(N,\F_p)^U$, i.e., $\res_{U,N}^1$ is surjective.

Set $w=[x_1,_{n-1}x_2]$, $z=[x_1,_{n-1}x_2]$.
Since $[z,x_2]=[w,_2x_2]=[x_1,_nx_2]\in N$, by Lemma~\ref{lem:comm} in $U$ one has
\begin{equation}\label{eq:rel UN G2}
 \begin{split}
[w,t]=[w,x_2^p] &= [w,x_2]^p\cdot[w,_2x_2]^{\binom{p}{2}}\cdots [w,_{p}x_2]\\
&= z^p\cdot [z,x_2]^{\binom{p}{2}}\cdots [z,_{p-1}x_2]\\
 &\equiv z^p \mod N.  \end{split} 
\end{equation}
Since $\chi_t,\chi_{1,n-2},\chi_{1,n-1}$ are linearly independent in $\Ker(\res_{U,N}^1)=\Img(\mathrm{inf}_{U/N,U}^1)$, by duality the quotient $U/N$ has a minimal generating set containing $tN,wN,zN$ as distinct elements.
Therefore, by \eqref{eq:rel UN G2} and by \cite[Thm.~8.1]{eq:kummer}, the pro-$p$ pair $\calG\vert_U/N$ is not Kummerian.

Finally, by Proposition~\ref{thm:kummer quot}, also the pro-$p$ pair $\calG\vert_U=(U,\theta\vert_U)$ is not Kummerian.
 \end{proof}

Proposition~\ref{prop:G2 UN}, together with Theorem~\ref{thm:1smooth Gal}, yields Theorem~\ref{thm:main} for $r_0$ of the first type. 

 \begin{cor}\label{cor:1st type}
Let $G=\langle\:x_1,\ldots,x_d\:\mid\:r_0,r_1,\ldots,r_\nu\:\rangle$ be as in Theorem~\ref{thm:main}, with the relation $r_0$ of the first type. 
Then $G$ does not occur as absolute Galois group.
\end{cor}

\begin{proof}
 By Proposition~\ref{prop:G2 UN}, $G$ cannot complete into a 1-smooth torsion-free pro-$p$ pair, and thus by Theorem~\ref{thm:1smooth Gal} $G$ does not occur as maximal pro-$p$ Galois group of a field $\K$ containing a primitive $p$-th root of 1 (and also $\sqrt{-1}$, if $p=2$).

 Now suppose that $p=2$.
 It is well-known that a field $\K$ contains $\sqrt{-1}$ if and only if for every epimorphism $\varphi\colon G_{\K}(2)\twoheadrightarrow\Z/2\Z$ there exists an epimorphism  $\tilde\varphi\colon G_{\K}(2)\twoheadrightarrow\Z/4\Z$ such that $\varphi=\pi\circ\tilde\varphi$, where $\pi\colon\Z/4\Z\twoheadrightarrow\Z/2\Z$ is the canonical projection
 (see, e.g., \cite[\S~1.2]{serre:topics}).
 This property holds for $G$, as $G/G'\simeq\Z/2^f\Z\times \Z_2^{d-1}$ with $f\geq2$, and thus $G$ cannot occur as maximal pro-2 Galois group of a field $\K$ which does not contain $\sqrt{-1}$.
 
Finally, by Remark~\ref{rem:GKquot} the pro-$p$ group $G$ does not occur as the absolute Galois group of any field.
\end{proof}



\subsection{The second type of $r_0$}\label{ssec:G1}

In this subsection, $G$ denotes a pro-$p$ group with minimal presentation $ G=\langle\: x_1,\ldots,x_d\:\mid\:r_0,r_1\ldots,r_\nu\:\rangle$ as in Theorem~\ref{thm:main}, with $r_0$ of the second type, i.e., $d$ is odd and
\[
 r_0=x_1^q\cdot[x_1,_nx_2]\cdot [x_3,x_4]\cdots[x_{d-1},x_d]\cdot s,
\]
with $s\in S''$, where $S$ is the subgroup of $G$ generated by $x_3,\ldots,x_d$; while $r_1,\ldots,r_\nu\in S'$.
Our goal is to show that for any orientation $\theta\colon G\to1+p\Z_p$, the pair $\calG=(G,\theta)$ is not 1-smooth.

First, we establish what a suitable orientation $\theta\colon G\to1+p\Z_p$ should be like.

\begin{prop}\label{prop:theta 1}
 Let $\theta\colon G\to1+p\Z_p$ be an orientation of $G$ such that the cyclotomic pro-$p$ pair $\calG=(G,\theta)$ is torsion-free. 
 If $\calG$ is Kummerian, then $\theta(x_i)=1$ for $i=3,\ldots,d$, and either
 \begin{itemize}
  \item[(i)] $\theta(x_1)=1$ and $(\theta(x_2)^{-1}-1)^n=- q$; or
  \item[(ii)] $\theta(x_1)\not=1$ and $\theta(x_2)=1$ (this is possible only if $ q=0$).
 \end{itemize}
In particular, if $-q\notin p^n\Z_p$, then $\calG$ is not Kummerian for any orientation $\theta\colon G\to1+p\Z_p$.
\end{prop}

\begin{proof}
First, note that $\theta\vert_{S'}\equiv\mathbf{1}$, and thus by Lemma~\ref{lem:1cocyc} one has $c\vert_{S''}\equiv\mathbf{0}$ for every continuous 1-cocycle $c\colon G\to\Z_p(\theta)$.
 From
 \[      1=\theta(1) = \theta\left(x_1^ q[x_1,_nx_2][x_3,x_4]\cdots[x_{d-1},x_d]\cdot s\right) =\theta(x_1)^ q
   \]
we deduce that $\theta(x_1)^ q=1$.
Since $\Img(\theta)$ is torsion-free, if $ q\neq0$ one has $\theta(x_1)=1$, while one might have $\theta(x_1)\neq1$ only if $ q=0$.

For every $i=1,\ldots,d$, let $c_i\colon G\to\Z_p(\theta)$ be the continuous 1-cocycle such that $c_i(x_j)=\delta_{ij}$, with $j\in\{1,\ldots,d\}$ (such a continuous 1-cocycle always exists by Proposition~\ref{prop:kummer 1cocyc}, since we are assuming that $\calG$ is Kummerian).
Then 
\begin{equation}\label{eq:1cocyc G1}
 \begin{split}
 0 &=c_i\left(x_1^ q[x_1,_nx_2][x_3,x_4]\cdots[x_{d-1},x_d]\right)\\
 &= c_i(x_1^ q)+\theta(x_1)^ q\cdot c_i\left([x_1,_nx_2][x_3,x_4]\cdots[x_{d-1},x_d]\right) \\
 &= c_i(x_1^ q)+c_i([x_1,_nx_2])+c_i([x_3,x_4])+\ldots+c_i([x_{d-1},x_d]).
 \end{split} 
\end{equation}
For $i\geq3$, $i$ odd, from \eqref{eq:1cocyc}, Lemma~\ref{lem:1cocyc}, and \eqref{eq:1cocyc G1}, we get
\[
0= 0+\ldots+c_i([x_i,x_{i+1}])+\ldots+0=\theta(x_ix_{i+1})^{-1}(1-\theta(x_{i+1}))\cdot1,
 \]
and hence $\theta(x_{i+1})=1$.
Analogously, for $i\geq4$, $i$ even, one gets 
\[0= 0+\ldots+c_i([x_{i-1},x_{i}])+\ldots+0=\theta(x_{i-1}x_i)^{-1}(-1+\theta(x_{i-1}))\cdot1,\]
and hence $\theta(x_{i-1})=1$.
On the other hand, for $i=1$ from \eqref{eq:1cocyc}, Lemma~\ref{lem:1cocyc} and \eqref{eq:1cocyc G1} we get
\[ 0=c_1(x_1^ q)+c_1([x_1,_nx_2])+0=\begin{cases}
                                           q+(\theta(x_2)^{-1}-1)^n & \text{if }\theta(x_1)=1, \\
                                          0+\theta(x_1)^{-1}(\theta(x_2)^{-1}-1)^n & \text{if }\theta(x_1)\not=1
                                         \end{cases}\]
(recall that in the second case necessarily $ q=0$): in the first case one deduces (i), in the second case one deduces (ii).
\end{proof}


Note that if $\calG=(G,\theta)$ is a Kummerian pro-$p$ pair, then $S\subseteq \Ker(\theta)$, and thus Lemma~\ref{lem:1cocyc} implies that $c\vert_{S'}\equiv\mathbf{0}$ for every continuous 1-cocycle $c\colon G\to\Z_p(\theta)$.
From now on we assume that $-q=\lambda^n$ for some $\lambda\in p\Z_p$, so that $G$ may be completed into a Kummerian pro-$p$ pair $\calG=(G,\theta)$.

Let $\phi\colon G\to\Z/p\Z$ be the homomorphism defined by $x_i\mapsto 0$ for $i=1,3,4,\ldots,d$ and $x_2\mapsto 1$, and set $U=\Ker(\phi)$.

\begin{lem}\label{lem basis U1}
 The set
 \[
\mathcal{X}_U=\left\{\begin{array}{c} t:=x_2^p, \\ 
x_1,[x_1,x_2],\ldots,[x_1,_{m-1}x_2], \\x_3,[x_3,x_2],\ldots,[x_3,_{p-1}x_2],\\ \vdots \\ x_d,[x_d,x_2],\ldots,[x_d,_{p-1}x_2]
\end{array}\right\} , 
 \]
 where $m=\min\{n,p\}$, is a minimal generating set of $U$.
\end{lem}

\begin{proof}
 Set $\bar U=U/U'$, and for $g\in U$ let $\bar g=g\cdot U'$ denote its image in $\bar U$.
 In particular, we set $v_{i,0}=\bar x_i$ and $v_{i,k}=\overline{[x_i,_kx_2]}\in \bar U$, for $i=1,\ldots,d$ and $k\geq1$.
   
  The group $U$ is generated as normal subgroup of $G$ by $t$ and by $x_1,x_3,x_4,\ldots,x_d$.
 Since $G'\subseteq U$, one has $G''\subseteq  U'$.
 Also, $S'\subseteq U'$, as $S\subseteq U$.
 By Lemma~\ref{lem:comm}, for every $i=1,3,\ldots, d$ one has $$[x_i,t]\equiv \left(\prod_{k=1}^{p-1}[x_i,_kx_2]^{\binom{p}{k}}
 \right)\cdot[x_i,_px_2]\mod U',$$
 and since $[x_i,t]\in U'$, it follows that for every $k\geq p$
  \begin{equation}\label{eq:commx1x2p G1}
 v_{i,k}=v_{i,1}^{\lambda_1}\cdot v_{i,2}^{\lambda_2}\cdots v_{i,p-2}^{\lambda_{p-2}}\cdot v_{i,p-1}^{\lambda_{p-1}}
 \end{equation}
 for some $\lambda_1,\ldots,\lambda_{p-1}\in p\Z_p$.
 Hence, $\bar U$ is generated as abelian pro-$p$ group by the set $\{\bar t, v_{i,k}\}_{i,k}$, with $i=1,3,\ldots,d$ and $0\leq k\leq p-1$.

In $U$, from the defining relation $r_0=1$ one obtains $p$ relations
\begin{equation}\label{eq:U2 rel mod comm}
 \begin{split}
r_0=x_1^q[x_1,_nx_2][x_3,x_4]\cdots[x_{d-1},x_d] &=1,\\
 [r_0,x_2]=\left[ x_1^q[x_1,_nx_2][x_3,x_4]\cdots[x_{d-1},x_d],x_2\right] &=1,\\
 &\vdots\\  [r_0,_{p-1}x_2]=\left[ x_1^q[x_1,_nx_2][x_3,x_4]\cdots[x_{d-1},x_d],_{p-1}x_2\right] &=1.
 \end{split}
\end{equation}
From the relations \eqref{eq:U2 rel mod comm} one deduces 
\begin{equation}\label{eq:rel U1}
 \left[x_1^q\cdot [x_1,_nx_2],_{k}x_2\right]\equiv [x_1,_kx_2]^q\cdot [x_1,_{n+k}x_2]\equiv 1\mod U'
\end{equation}
for every $0\leq k\leq p-1$.
Note that one obtains the abelian pro-$p$ group $\bar U$ as the quotient of the $\Z_p$-module $A$ --- where $A$ is the free $\Z_p$-module generated by $\{\bar t,v_{i,k}\}$, with $i=1,3,\ldots,d$ and $0\leq k\leq p-1$ --- over the $p$ relations \eqref{eq:rel U1}, as the relations $r_1,\ldots,r_\nu$ of $G$ lie in $S'\subseteq U'$, and this yields an epimorphism of abelian pro-$p$ groups $\pi\colon A\twoheadrightarrow \bar U$.
Moreover, one has $\bar U/\bar U^p=\bar U/\Phi(\bar U)\simeq U/\Phi(U)$.

Suppose first that  $n\geq p$.
By applying \eqref{eq:commx1x2p G1} to the $p$ relations \eqref{eq:rel U1} one obtains $p$ relations in $\bar U$
\[
v_{1,k}^{-q}= v_{1,n+k}=\prod_{j=1}^{p-1}v_{1,j}^{\lambda_j}, \qquad \text{for }k=0,\ldots,p-1,
\]
for some $\lambda_1,\ldots,\lambda_{p-1}\in p\Z_p$ depending on $k$.
Since $p\mid q$, one has $\Ker(\pi)\subseteq pA$ (here we use the additive notation for $A$).
Therefore, the set $\{\bar t,v_{i,k}\}_{i,k}$, with $i=1,3,\ldots,d$ and $0\leq k<p$, is a minimal generating set of $\bar U$, and the isomorphism $U/\Phi(U)\simeq \bar U/\Phi(\bar U)$ implies that $\mathcal{X}_U$ generates minimally $U$.

Suppose now that $2<n<p$.
By applying \eqref{eq:commx1x2p G1} to the $p$ relations \eqref{eq:rel U1} one obtains the following $p$ relations in $\bar U$:
\begin{eqnarray}\label{eq:rel Ubar 2a}
 v_{1,n+k} = v_{1,k}^{-q}, &&\quad\text{for } k=0,\ldots,p-n-1, \\
v_{1,k}^{-q} = v_{1,n+k}=\prod_{j=1}^{p-1}v_{1,j}^{\lambda_j},&& \quad \text{for } k=p-n,\ldots,p-1,\label{eq:rel Ubar 2a}
\end{eqnarray}
for some $\lambda_1,\ldots,\lambda_{p-1}\in p\Z_p$ depending on $k$.
Consider the quotient $\bar U/\bar U^p=\bar U/\Phi(\bar U)$, which is isomorphic to $U/\Phi(U)$.
The $p-n$ relations \eqref{eq:rel Ubar 2a} translate into
\[
 v_{1,n+k} \equiv 1\mod \bar U^p, \qquad\text{for } k=0,\leq,p-n-1,
\]
while the $n$ relations \eqref{eq:rel Ubar 2a} become trivial in $\bar U/\bar U^p$.
Therefore, the cosets 
$$\left\{\bar t\bar U^p, v_{1,k}\bar U^p, v_{i,k'}\bar U^p\:\mid\:i=3,\ldots,d;\: 0\leq k\leq n-1;\: 0\leq k'\leq p-1\right\}$$
generate minimally the $p$-elementary abelian pro-$p$ group $\bar U/\bar U^p\simeq U/\Phi(U)$, and thus $\mathcal{X}_U$ generates minimally $U$.
\end{proof}

\begin{rem}\label{rem:comm NG}\rm
Suppose that the pro-$p$ pair $\calG=(G,\theta)$ is Kummerian.
Since $G'\subseteq\Ker(\theta\vert_U)$, one has $G''\subseteq\Ker(\theta\vert_U)'$.
If $x_1\in\Ker(\theta\vert_U)$, then for every $m\geq1$ one has
\begin{equation}\label{eq:commx1}
\begin{split}
  [x_1^q,_mx_2]&=\left[[x_1,x_2]^{x_1^{q-1}}\cdots[x_1,x_2]^{x_1}\cdot[x_1,x_2],_{m-1}x_2\right]\\
  &\equiv\left[[x_1,x_2]^q,_{m-1}x_2\right]\mod \Ker(\theta\vert_U)'\\
  &\equiv [x_1,_{m}x_2]^q \mod \Ker(\theta\vert_U)',
\end{split}
\end{equation}
as $y^{x_1}\equiv y\bmod \Ker(\theta\vert_U)'$ for every $y\in \Ker(\theta\vert_U)$.
On the other hand, if $x_1\notin \Ker(\theta\vert_U)$ then $q=0$ by Proposition~\ref{prop:theta 1}, and \eqref{eq:commx1} holds trivially.
\end{rem}

\begin{prop}\label{prop:G1 UN}
Suppose that $\calG=(G,\theta)$ is a Kummerian torsion-free pro-$p$ pair, and set $N=\Ker(\theta\vert_U)'$. 
 Then the pro-$p$ pairs $\calG\vert_U/N=(U/N,\bar\theta\vert_U)$ and $\calG\vert_U=(U,\theta\vert_U)$ are not Kummerian.
\end{prop}

\begin{proof}
Recall that $G''\subseteq N$ and $S'\subseteq N$, and that the commutaturs $[x_3,x_4],\ldots,[x_{d-1},x_d]$ lie in $N$.
Moreover, one has $p^n\mid q$, by Proposition~\ref{prop:theta 1}.

Set $w=[x_1,_{n-2}x_2]$ and $z=[x_1,_{n-1}x_2]=[w,x_2]$.
Then by Lemma~\ref{lem basis U1}, $t,w,z$ are distinct elements of the minimal generating set $\mathcal{X}_U\subseteq U$. 
Since $[z,x_1]\equiv x_1^{-q}\bmod N$, by Lemma~\ref{lem:comm} and by Remark~\ref{rem:comm NG} one has
\begin{equation}\label{eq:wt1}
 \begin{split} 
 [w,t] &\equiv [w,x_2]^p[w,_2x_2]^{\binom{p}{2}}\cdots [w,_{p-1}x_2]^{p}[w,_px_2]\mod N\\
  &\equiv z^p[z,x_2]^{\binom{p}{2}}[z,_2x_2]^{\binom{p}{3}}\cdots [z,_{p-2}x_2]^{p}[z,_{p-1}x_2]\mod N\\
 &\equiv z^p \cdot x_1^{-q\binom{p}{2}}\cdot[x_1^{-q},x_2]^{\binom{p}{3}}\cdots [x_1^{-q},_{p-3}x_2]^{p}[x_1^{-q},_{p-2}x_2]\mod N\\
  &\equiv z^p \cdot x_1^{-q\binom{p}{2}}[x_1,x_2]^{-q\binom{p}{3}}\cdots [x_1,_{p-3}x_2]^{-qp}[x_1,_{p-2}x_2]^{-q}\mod N.
 \end{split}
\end{equation}
Now we have three cases, depending on the value of $n$.
\begin{itemize}
 \item[(a)]  If $n\geq p$ then from \eqref{eq:wt1} one has
\[ z^{p}\equiv[w,t]\cdot x_1^{q\binom{p}{2}}[x_1,x_2]^{q\binom{p}{3}}\cdots [x_1,_{p-3}x_2]^{qp}w^{q}\mod N.
\]
\item[(b)] If $n=p-1$, then from \eqref{eq:wt1} one has
\[  
   [w,t] \equiv z^p\cdot x_1^{-q\binom{p}{2}}\cdot[x_1,x_2]^{-q\binom{p}{3}}\cdots w^{-qp}\cdot z^{-q}\mod N,
 \]
and thus 
\[  z^{p-q}\equiv [w,t]\cdot x_1^{q\binom{p}{2}}\cdot[x_1,x_2]^{q\binom{p}{3}}\cdots[x_1,_{n-3}x_2]^{q\binom{p}{p-2}}\cdot w^{qp}\mod N.\]
Since $p^n\mid q$, one has $p-q\neq0$.
\item[(c)] If $n<p-1$, after iterating the replacement $[x_1,_kx_2]\equiv[x_1,_{k-n}x_2]^{-q}\bmod N$ in the last line of \eqref{eq:wt1} whenever $k\geq n$, one obtains 
\[
 z^{p+\lambda_{n-1}}\equiv [w,t]\cdot x_1^{\lambda_0}\cdot[x_1,x_2]^{\lambda_1}\cdots w^{\lambda_{n-1}} \mod N,
\]
with $\lambda_0,\ldots,\lambda_{n-1}\in p^2\Z_p$, as $p^2\mid q$.
In particular, $p+\lambda_{n-1}\neq0$.
\end{itemize}
In all three cases, \cite[Thm.~8.1]{eq:kummer} implies that the pro-$p$ pair $\calG\vert_U/N$ cannot be Kummerian.

Finally, since $N\subseteq K(\calG\vert_U)$, by Theorem~\ref{thm:kummer} also the pro-$p$ pair $\calG\vert_U=(U,\theta\vert_U)$ cannot be Kummerian.
\end{proof}

Proposition~\ref{prop:G1 UN}, together with Theorem~\ref{thm:1smooth Gal}, yields Theorem~\ref{thm:main} for $r_0$ of the second type. 

 \begin{cor}
Let $G=\langle\:x_1,\ldots,x_d\:\mid\:r_0,r_1,\ldots,r_\nu\:\rangle$ be as in Theorem~\ref{thm:main}, with the relation $r_0$ of the second type. 
Then $G$ does not occur as absolute Galois group.
\end{cor}

\begin{proof}
 The proof is verbatim the same as the proof of Corollary~\ref{cor:1st type}.
\end{proof}



\section{One-relator pro-$p$ groups that are not absolute Galois groups}\label{sec:onerel}

In the next two sections we focus on those pro-$p$ groups defined in Theorem~\ref{thm:main} with a single defining relation --- i.e., $\nu=0$.


\subsection{Quadratic cohomology}\label{ssec:quadcohom}

Let $G$ be a one-relator pro-$p$ group of one of the two families defined in Theorem~\ref{thm:main}.
Then one has either
\[ \begin{split}
  r_0&\equiv [x_2,x_3]\cdots[x_{d-1},x_d]\mod G_{(3)},   \qquad\text{or}\\
  r_0&\equiv [x_3,x_4]\cdots[x_{d-1},x_d]\mod G_{(3)}.
 \end{split}\]
Hence, the $\F_p$-cohomology algebra $H^\bullet(G,\F_p)$ of $G$ is quadratic by Proposition~\ref{prop:onerel}. 
In particular, let $\mathcal{B}=\{\chi_1,\ldots,\chi_d\}$ be a basis of $H^1(G,\F_p)$ dual to $\mathcal{X}=\{x_1,\ldots,x_d\}$.
Then one has the following.
\begin{itemize}
 \item[(a)] If $r_0$ is of the first type, then
 \begin{equation}  \label{eq:cohom 1}
 \begin{split}
 &\chi_2\cup\chi_3=\ldots=\chi_{d-1}\cup\chi_d\neq0, \\
 &\chi_i\cup\chi_j\qquad \text{for }1\leq i<j\leq d, (i,j)\neq(2,3),\ldots,(d-1,d),
\end{split} \end{equation}
and $\chi_2\cup\chi_3$ generates $H^2(G,\F_p)$.
In particular 
\[
 H^1(G,\F_p)^\perp=\{\chi\in H^1(G,\F_p)\:\mid\: \chi\cup\psi=0\;\forall\:\psi\in H^1(G,\F_p)\}=\mathrm{Span}_{\F_p}(\chi_1).
\]
\item[(b)] If $r_0$ is of the second type, then
 \begin{equation} \label{eq:cohom 2}
\begin{split}
 &\chi_3\cup\chi_4=\ldots=\chi_{d-1}\cup\chi_d\neq0, \\
 &\chi_i\cup\chi_j\qquad \text{for }1\leq i<j\leq d, (i,j)\neq(3,4),\ldots,(d-1,d),
\end{split} \end{equation}
and $\chi_3\cup\chi_4$ generates $H^2(G,\F_p)$.
In particular 
\[
 H^1(G,\F_p)^\perp=\{\chi\in H^1(G,\F_p)\:\mid\: \chi\cup\psi=0\;\forall\:\psi\in H^1(G,\F_p)\}=\mathrm{Span}_{\F_p}(\chi_1,\chi_2).
\]
\end{itemize}
This proves Proposition~\ref{prop:intro}--(i).

\begin{rem}\label{rem:cd2 torfree}\rm
 Since $H^k(G,\F_p)=0$ for $k\geq3$, the pro-$p$ group $G$ is torsion-free (cf., e.g., \cite[Prop.~A]{qsv:quad}).
\end{rem}


\subsection{Free and abelian subgroups}\label{ssec:subgroups}
Recall from Proposition~\ref{prop:BK split} that Bloch-Kato pro-$p$ groups contain a unique maximal normal abelian subgroups, and that if a Bloch-Kato pro-$p$ group is not meta-abelian, then it contains a free non-abelian subgroup.
We prove that the one-relator pro-$p$ groups defined in Theorem~\ref{prop:intro} satisfy these two group-theoretic properties.

\begin{prop}
Set $G=\langle\: x_1,\ldots,x_d\:\mid\: r_0\:\rangle$ as in Theorem~\ref{prop:intro}. Then
\begin{itemize}
 \item[(i)] $G$ contains a free non-abelian subgroup;
 \item[(ii)] $G$ contains no normal abelian subgroups.
\end{itemize}
\end{prop}

\begin{proof}
By \cite[Prop.~4.10]{cq:onerel} one has a short exact sequence of pro-$p$ groups
\begin{equation}\label{eq:ses onerel}
 \xymatrix{\{1\}\ar[r] & N\ar[r] & G\ar[r] & G/N\ar[r] & \{1\}}
\end{equation}
where $N$ is the normal subgroup of $G$ generated by $x_1$, respectively by $x_1,x_2$, depending on the family of $G$,
such that $N$ is a free non-abelian pro-$p$ group, and $G/N$ is a torsion-free Demushkin group on $d-1$, respectively $d-2$, generators.
This proves (i).

Suppose that $A$ is a non-trivial normal abelian subgroup of $G$.
Since $G$ is torsion-free by Remark~\ref{rem:cd2 torfree}, $A$ is a free abelian pro-$p$ group, i.e., $A$ is isomorphic to the direct product of (possibly infinite) copies of $\Z_p$.
Since $N$ is free and non-abelian, one has $A\cap N=\{1\}$, and thus
\begin{equation}\label{eq:A in Gbar}
 A=A/(A\cap N)\simeq AN/N\unlhd  G/N.
\end{equation}
It is well-known that a torsion-free Demushkin group has a non-trivial normal abelian subgroup if and only if it has 2 generators.
Thus, 
$$G/N\simeq\langle\,\bar x_i,\bar x_{i+1}\mid [\bar x_i,\bar x_{i+1}]=1\,\rangle\simeq\Z_p^2,$$
where $i=2$, respectively $i=3$, and $\bar x_j=x_jN$ denotes the image of $x_j$ in $G/N$, for $j=i,i+1$.
In particular, $A\simeq\Z_p^a$, with $a\in\{1,2\}$, since $A\unlhd G/N$ by \eqref{eq:A in Gbar}.

Since $G/N$ is abelian and $A$ is normal in $G$, one has
\[
N\supseteq G'\supseteq [G,A]\subseteq A,                                                          \]
and since $A\cap N=\{1\}$, one deduces $[G,A]=\{1\}$, i.e., $A$ is contained in the center $\mathrm{Z}(G)$ of $G$.
Hence, for every $x\in G\smallsetminus A$, the subgroup $\langle x\rangle A$ of $G$ is abelian and torsion-free, so that $\langle x\rangle A\simeq \Z_p^{a+1}$.
In particular, if $x^{p^k}\in A$ for some $x\in G$ and $k\geq1$, then $x\in A$, and consequently $G/A$ is a torsion-free pro-$p$ group.
Since the $\F_p$-cohomology algebra $H^\bullet(G,\F_p)$ is quadratic, from \cite[Prop.~7.11]{qw:cyc} one deduces that $A\not\subseteq \Phi(G)$, and $G\simeq A\times(G/A)$.
Then by \cite[Thm.~2.4.6]{nsw:cohn} one has 
\[\begin{split}
   H^2(G,\F_p) &\simeq\bigoplus_{s+t=2}H^s(G/A,H^t(A,\F_p))\\ &\simeq H^2(G/A,\F_p)\oplus H^1(G/A,\F_p)\oplus H^2(A,\F_p),
  \end{split}\]
as $G/A$ acts trivially on $H^t(A,\F_p)$, $t=1,2$, and this contradicts \eqref{eq:cohom 1}--\eqref{eq:cohom 2}.
Therefore, $A=\{1\}$, and this proves (ii).
\end{proof}

Therefore, one may not employ Proposition~\ref{prop:BK split} to show that the pro-$p$ groups of the two families defined in Theorem~\ref{thm:main} are not Bloch-Kato --- and thus that they cannot occur as absolute Galois groups of fields.
We ask the following.

\begin{question}
Are (some of) the one-relator pro-$p$ groups defined in Theorem~\ref{prop:intro} Bloch-Kato pro-$p$ groups?
\end{question}

\subsection{Benson-Lemire-Minac-Swallow's examples}\label{ssec:BLMS}

For $x,y$ elements of a (pro-$p$) groups $G$ and $n\geq1$, set
\[
  [x\,{}_n,y]=[\underbrace{x,\ldots,[x[x}_{n\text{ times}},y]]\ldots].
\]
Let $p$ be odd, and let $G$ be a finitely generated one-relator pro-$p$ group $G$ with minimal presentation
$ G=\langle\:y_1,\ldots,y_d\:\vert\:t\:\rangle$, $d\geq4$, and defining relation
\[
 t=y_1^q\cdot [y_1\,{}_n,y_2][y_3,y_4]\cdots[y_{m-1},y_m]\cdot [y_1^p,y_{k_1}]\cdots[y_1^p,y_{k_\nu}]
\]
where:
\begin{itemize}
 \item[(a)] $q=p^f$ with $f\in\{1,2,\ldots,\infty\}$ and $2\leq n\leq p-1$;
 \item[(b)] $4\leq m\leq d$ is even;
 \item[(c)] $\nu\geq0$ and $3\leq k_1<\ldots<k_\nu\leq d$, if $\nu\neq0$;
 \item[(d)] $\{3,\ldots,m\}\cup\{k_1,\ldots,k_\nu\}=\{3,\ldots,d\}$.
\end{itemize}
This family of one-relator pro-$p$ groups was introduced by Benson et al. in \cite{BLMS} --- for short, we shall call these pro-$p$ groups {\sl BLMS pro-$p$ groups}.
In \cite[Cor.~1]{BLMS} it is shown that BLMS pro-$p$ groups do not occur as maximal pro-$p$ Galois groups of fields containing a primitive $p$-th root of 1.
We show that the intersection of the family of BLMS pro-$p$ groups with the family of one-relator pro-$p$ groups defined in Theorem~\ref{prop:intro} is rather small.

\begin{prop}\label{prop:BLMS}
 Let $G=\langle\:x_1,\ldots,x_d\:\vert\: r_0\:\rangle$ be a one-relator pro-$p$ group as defined in Theorem~\ref{prop:intro}, and suppose that $G$ may complete into a Kummerian pro-$p$ pair $\calG=(G,\theta)$.
 Then $G$ is a BLMS pro-$p$ group only if 
\begin{equation}\label{eq:rel prop BLMS statement} r_0=[x_1,_nx_2][x_3,x_4]\cdots[x_{d-1},x_d]\cdot s.\end{equation}
\end{prop}

\begin{proof}
Suppose that $G$ is a BLMS pro-$p$ group.
Thus, $G=\langle\:y_1,\ldots,y_d\:\vert\:t\:\rangle$ with $t$ as above.
Since 
\[
 t\equiv [y_3,y_4]\cdots[y_{m-1},y_m]\mod G_{(3)},
\]
by Proposition~\ref{prop:onerel} the $\F_p$-cohomology algebra $H^\bullet(G,\F_p)$ is quadratic, and in particular one has 
$\dim(H^1(G,\F_p)^\perp)=2+d-m$.
Hence, by \eqref{eq:cohom 1}--\eqref{eq:cohom 2} one has that $\dim(H^1(G,\F_p)^\perp)=2$, so that $r_0$ is necessarily of the second type, i.e.,
\begin{equation}\label{eq:rel prop BLMS}
  r_0=x_1^q[x_1,_nx_2][x_3,x_4]\cdots[x_{d-1},x_d]\cdot s,
\end{equation}
for some $s\in S''$, where $S$ is the subgroup of $G$ generated by $x_1,x_3,\ldots,x_d$.

We claim that in \eqref{eq:rel prop BLMS} one has $q=0$, so that $r_0$ is as stated in \eqref{eq:rel prop BLMS statement}.
For this purpose, let $\theta\colon G\to1+p\Z_p$ be an orientation such that the pro-$p$ pair $\calG=(G,\theta)$ is Kummerian.
Then one has $\theta(y_1)^q=\theta(t)=1$, and moreover Proposition~\ref{prop:kummer 1cocyc} yields $c(t)=0$ for every continuous 1-cocycle $c\colon G\to\Z_p(\theta)$.
For every $i=1,\ldots,d$, let $c_i\colon G\to\Z_p(\theta)$ be the continuous 1-cocycle such that $c_i(y_j)=\delta_{ij}$.
Then one has
\begin{equation}\label{eq:1cocyc BLMS}
\begin{split}
  0 &= c_i\left(y_1^q[y_1\,{}_n,y_2][y_3,y_4]\cdots[y_{d-1},y_d]\cdot [y_1^p,y_{k_1}]\cdots[y_1^p,y_{k_\nu}]\right)\\
  &= c_i(y_1^p)+c_i([y_1\,{}_n,y_2])+c_i([y_3,y_4])+\ldots+c_i([y_{d-1},y_d])+\sum_{j=1}^\nu c_i\left([y_1^p,y_{k_j}]\right).
\end{split}
\end{equation}
By Lemma~\ref{lem:1cocyc}, from \eqref{eq:1cocyc BLMS} one obtains
\[
 0=c_2(t)=c_2([y_1\,{}_n,y_2])=(1-\theta(y_1)^{-1})^n,
\]
and thus $\theta(y_1)=1$.
For $i\geq3$, $i$ odd, by Lemma~\ref{lem:1cocyc} from \eqref{eq:1cocyc BLMS} one obtains
\[
0=\begin{cases} \theta(y_iy_{i+1})^{-1}(1-\theta(y_{i+1})), &\text{if }i\neq k_j\;\forall j=1,\ldots,\nu,\\
    \theta(y_iy_{i+1})^{-1}(1-\theta(y_{i+1}))+\theta(y_i)^{-1}(1-\theta(y_1)^p), &\text{if }i\in\{k_1,\ldots,k_\nu\},
   \end{cases}\]
and in both cases one has $\theta(y_{i+1})=1$, as $\theta(y_1)=1$.
Analogously, one obtains $\theta(y_{i-1})=1$ for $i\geq3$, $i$ even.
Finally, 
\[0 =c_1(x_1^q)+\sum_{j=1}^\nu c_1\left([y_1^p,y_{k_j}]\right) = q+\sum_{j=1}^\nu (\theta(y_{k_j})^{-1}-1)\cdot p= q+0,\]
and hence $q=0$.
\end{proof}

Therefore, most of the one-relator pro-$p$ groups defined in Theorem~\ref{prop:intro} are not BLMS pro-$p$ groups.


\section{Massey products}\label{sec:massey}

\subsection{Massey products and pro-$p$ groups}\label{ssec:massey}

Let $G$ be a pro-$p$ group.
For $m\geq2$, the {\sl $m$-fold Massey product} on $H^1(G,\F_p)$ is a multi-valued map
\[
 \underbrace{H^1(G,\F_p)\times \ldots\times H^1(G,\F_p)}_{m\text{ times}}\longrightarrow H^2(G,\F_p).
\]
For $m\geq2$ elements $\psi_1,\ldots,\psi_m$ of $H^1(G,\F_p)$, we write $\langle\psi_1,\ldots,\psi_m\rangle$ for the set of values of the $m$-fold Massey product of the elements $\psi_1,\ldots,\psi_m$.
If $m=2$, then the 2-fold Massey product coincides with the cup-product, i.e., for $\psi_1,\psi_2\in H^1(G,\F_p)$ one has 
$$\langle\psi_1,\psi_2\rangle=\{\psi_1\cup\psi_2\}\subseteq H^2(G,\F_p).$$
For further details on this operation in the general homological context, we direct the reader to \cites{kraines:massey,dwier:massey}, 
and to \cites{vogel,ido:massey,MT:masseyGal} for Massey products in the profinite and Galois-theoretic context.
In particular, the definition of $m$-fold Massey products in the $\F_p$-cohomology of pro-$p$ groups may be found in \cite[Def.~2.1]{MT:masseyGal}.
For the purposes of our investigation, the group-theoretic conditions given by Proposition~\ref{prop:masse unip} below will be enough.

Given $m\geq2$ elements $\psi_1,\ldots,\psi_m\in H^1(G,\F_p)$, the Massey product $\langle\psi_1,\ldots,\psi_m\rangle$ is said to be {\sl defined} if it is non-empty, and it is said to {\sl vanish} if it contains 0.
In the following proposition we collect some properties of Massey products (cf., e.g., \cite[\S~1.2]{vogel} and  \cite[\S~2]{MT:masseyGal}).

\begin{prop}\label{prop:massey cup}
 Let $G$ be a pro-$p$ group and let $\psi_1,\ldots,\psi_m\in H^1(G,\F_p)$, $m\geq3$.
 \begin{itemize}
 \item[(i)] If $\psi_k=0$ for some $k\in\{1,\ldots,m\}$ and the Massey product $\langle\psi_1,\ldots,\psi_m\rangle$ is defined, then $0\in\langle\psi_1,\ldots,\psi_m\rangle$.
\item[(ii)] If the Massey product $\langle\psi_1,\ldots,\psi_m\rangle$ is defined, then $\psi_i\cup\psi_{i+1}=0$ for all $i=1,\ldots,m-1$.
 \item[(iii)] If the Massey product $\langle\psi_1,\ldots,\psi_m\rangle$ is defined, then one has the inclusion 
 \begin{equation}\label{eq:massey coset}
  \langle\psi_1,\ldots,\psi_m\rangle\supseteq \xi+\psi_1\cup H^1(G,\F_p) +\psi_m\cup H^1(G,\F_p),\end{equation}
  where the right-side term of \eqref{eq:massey coset} is the set $\{\xi+\psi_1\cup\varphi +\psi_m\cup\varphi'\mid \varphi,\varphi'\in H^1(G,\F_p)\}$, and $\xi$ is any element of $\langle\psi_1,\ldots,\psi_m\rangle$; if $m=3$ then \eqref{eq:massey coset} is an equality.
 \end{itemize}
 \end{prop}


\subsection{Massey products and unipotent representations}\label{ssec:unipotent}

Massey products for a pro-$p$ group $G$ may be interpreted in terms of unipotent upper-triangular representations of $G$ as follows.
For $m\geq 2$ let
\[
 \dbU_{m+1}=\left\{\left(\begin{array}{ccccc} 1 & a_{1,2} & \cdots & & a_{1,m+1} \\ & 1 & a_{2,3} &  \cdots & \\
 &&\ddots &\ddots& \vdots \\ &&&1& a_{m,m+1} \\ &&&&1 \end{array}\right)\mid a_{i,j}\in\F_p \right\}\subseteq 
 \mathrm{GL}_{m+1}(\F_p)
\]
be the group of unipotent upper-triangular $(m+1)\times(m+1)$-matrices over $\F_p$.
Then $\dbU_{m+1}$ is a $p$-group.
Moreover, let $I_{m+1},E_{i,j}\in \dbU_{m+1}$ denote respectively the identity $(m+1)\times(m+1)$-matrix and the $(m+1)\times(m+1)$-matrix with 1 at entry $(i,j)$ and 0 elsewhere, for $1\leq i<j\leq m+1$.
The group $\dbU_{m+1}$ has the following properties.

\begin{lem}\label{lem:Um+1}
 \begin{itemize}
  \item[(i)] The center $\mathrm{Z}(\dbU_{m+1})$ is the subgroup 
  $$I_{m+1}+E_{1,m+1}\F_p=\left\{\:I_{m+1}+a\cdot E_{1,m+1}\:\mid\: a\in\F_p\:\right\}.$$
  \item[(ii)] The exponent of $\dbU_{m+1}$ is $p^{\lfloor \log_p(m+1)\rfloor+1}$.
  \item[(iii)] The $(m+1)$-th element $\gamma_{m+1}(\dbU_{m+1})$ of the descending central series of $\dbU_{m+1}$ is trivial.
 \end{itemize}
\end{lem}

\begin{proof}
 For (i) see, e.g., \cite[\S~3]{MT:masseyGal} or \cite[p.~308]{eq:kummer}.
 For (ii) see \cite[Prop.~2.3]{ido:series}.
 Item~(iii) may be deduced from \cite[Lemma~3.8]{MT:masseyGal}.
\end{proof}

One has the following group-theoretic result to check whether a given $m$-fold Massey product of a pro-$p$ group $G$ is defined, or vanishes (see, e.g., \cite[Lemma~9.3]{eq:kummer}).

\begin{prop}\label{prop:masse unip}
Let $G$ be a pro-$p$ group and let $\psi_1,\ldots,\psi_m\in H^1(G,\F_p)$, $m\geq2$. Set $\bar \dbU_{m+1}=\dbU_{m+1}/\mathrm{Z}(\dbU_{m+1})$.
\begin{itemize}
 \item[(i)] The $m$-fold Massey product $\langle\psi_1,\ldots,\psi_m\rangle$ is defined if and only if there exists a continuous homomorphism $\bar \rho\colon G\to\bar\dbU_{m+1}$ such that $\bar\rho_{i,i+1}=\psi_i$ for every $i=1,\ldots,m$.
 \item[(ii)] The $m$-fold Massey product $\langle\psi_1,\ldots,\psi_m\rangle$ vanishes if and only if there exists a continuous homomorphism $ \rho\colon G\to\dbU_{m+1}$ such that $\rho_{i,i+1}=\psi_i$ for every $i=1,\ldots,m$.
\end{itemize}
(Here $\bar\rho_{i,i+1}$ and $\rho_{i,i+1}$ denote the projection of $\bar\rho$, resp. $\rho$, on the $(i,i+1)$-entry.)
\end{prop}

Let $\K$ be a field containing a primitive $p$-th root of 1. 
 In \cite[Conj.~1.1]{MT:masseyall}, Minac and T\^an conjectured that every defined $m$-fold Massey product on $H^1(G_{\K}(p),\F_p)$ vanishes, for every $m\geq3$.
 Roughly speaking, this conjecture says that if $\K$ is a field containing a primitive $p$-th root of 1, the $\F_p$-cohomology algebra $H^\bullet(G_{\K}(p),\F_p)$ should allow no operations which are external with respect to its quadratic algebra structure given by the cup-product.
This conjecture has been proved in the following cases:
\begin{itemize}
 \item[(a)] if $m=3$, by Efrat--Matzri, and independently by Minac--T\^an (cf. \cite{EM:massey} and \cite{MT:masseyall});
 \item[(b)] if $m=4$ and $\K$ is a local field containing also a primitive $p^2$-th root of 1, by Guillot--Minac (cf. \cite{PJ});
 \item[(c)] for every $m\geq3$ if $\K$ is a number field, by Harpaz--Wittenberg (cf. \cite{HW:massey}).
\end{itemize}

In \cite[\S~7]{MT:masseyGal}, Minac and T\^an produced some examples of pro-$p$ groups with defined and non-vanishing 3-fold Massey products, and hence which do not occur as maximal pro-$p$ Galois groups of fields containing a primitive $p$-th root of 1.

\begin{exam}\rm
Let $G$ be the pro-$p$ group with minimal presentation
\[
 G=\langle \:x_1,\ldots,x_5\:\mid\: [[x_1,x_2],x_3][x_4,x_5]\:\rangle.
\]
Then there is a $3$-fold Massey product on $H^1(G,\F_p)$ which is defined but does not vanishes (cf. \cite[Ex.~7.2]{MT:masseyGal}).
Note that by Proposition~\ref{prop:onerel}, the $\F_p$-cohomology algebra $H^\bullet(G,\F_p)$ is quadratic.
\end{exam}


\subsection{The first type of $r_0$}\label{ssec:masseyG2}

 We split the proof of Theorem~\ref{prop:intro}--(iia) in the following two propositions.
 
 \begin{prop}\label{prop:G2 massey 3}
   Let $G=\langle\:x_1,\ldots,x_d\:\mid \: r_0\:\rangle$ be a one-relator pro-$p$ group with $r_0$ of the first type as defined in Theorem~\ref{thm:main}.
 Then every defined 3-fold Massey product in $H^\bullet(G,\F_p)$ vanishes if and only if $q\neq3$.
 \end{prop}
 
 \begin{proof}
Pick three elements $\psi_1,\psi_2,\psi_3\in H^1(G,\F_p)$ (by Proposition~\ref{prop:massey cup}--(i), we may assume that they are all non-trivial), and write
\[
 \psi_1=\sum_{i=1}^da_i\chi_i,\qquad \psi_2=\sum_{i=1}^db_i\chi_i,\qquad \psi_3=\sum_{i=1}^dc_i\chi_i,\]
 with $a_i,b_i,c_i\in\F_p$.
 Then by Proposition~\ref{prop:onerel} one has
 \begin{equation}\label{eq:vanmassey G2 3}
    \psi_1\cup H^1(G,\F_p)+\psi_3\cup H^1(G,\F_p)=H^2(G,\F_p)\ni 0,
 \end{equation}
unless $a_i=c_i=0$ for every $i=2,\ldots,d$.
 
Suppose that the Massey product $\langle\psi_1,\psi_2,\psi_3\rangle$ is defined.
If some of the coefficients $a_i,c_i$ are not 0 for $i\in\{2,\ldots,d\}$, then by Proposition~\ref{prop:massey cup}--(iii) and by \eqref{eq:vanmassey G2 3} the Massey product $\langle\psi_1,\psi_2,\psi_3\rangle$ vanishes.

Assume that $q\neq3$ and $a_i=c_i=0$ for every $i=2,\ldots,d$. Set 
\[
 A=I_4+a_1E_{1,2}+b_1E_{2,3}+c_1E_{3,4}\in\dbU_4,\qquad B_i=I_4+b_iE_{2,3}\in\dbU_4\]
for $i=2,\ldots,d$.
By Lemma~\ref{lem:Um+1}--(ii), the exponent of $\dbU_4$ is $p^2$, if $p=2,3$, and $p$ if $p\geq5$.
Therefore, the exponent of $\dbU_4$ divides $q$ (recall that $4$ divides $q$ if $p=2$, and we are assuming now that $q\neq3$ if $p=3$), and $A^q=I_4$.
Also, by Lemma~\ref{lem:Um+1}--(c) one has $\gamma_4(\dbU_4)=\{1\}$.
Finally, one has $[A,_2B_2]=[B_i,B_{i+1}]=I_4$ for every $2\leq i\leq d-1$, $i$ even.
Therefore,
\begin{equation}\label{eq:rho r G2}
 A^ q\cdot \left[A,_nB_2\right]\cdot\left[B_2,B_3\right]\cdots\left[B_{d-1},B_d\right]=I_4.
 \end{equation}
 for any $n\geq3$.
Hence, the assignment $x_1\mapsto A$, $x_i\mapsto B_i$ for every $i=2,\ldots,d$ yields a homomorphism $\rho\colon G\to\dbU_4$, and thus $\langle\psi_1,\psi_2,\psi_3\rangle$ vanishes by Proposition~\ref{prop:masse unip}--(ii).

Conversely, assume that $q=3$, and pick $\psi_1,\psi_2,\psi_3\in H^1(G,\F_p)$ such that $a_1,b_1,c_1\neq0$ and $a_i=c_i=0$ for $i=2,\ldots,d$.
For every matrix $A=(a_{h,k})\in\dbU_4$, $B_i=(b(i)_{h,k})\in\dbU_4$ such that $a_{1,2}=a_1$, $a_{2,3}=b_1$, $a_{3,4}=c_1$, and $b(i)_{2,3}=b_i$, $b(i)_{1,2}=b(i)_{3,4}=0$ for all $i=2,\ldots,d$, 
one has 
\[
 [A,_2B_2]=[B_2,B_3]=\ldots=[B_{d-1},B_d]=I_4 \qquad\text{and}\qquad A^3=I_4+a_1b_1c_1E_{1,4}.
\]
Therefore, equality \eqref{eq:rho r G2} holds modulo $\mathrm{Z}(\dbU_4)$, and the assignment $x_1\mapsto \bar A$, $x_i\mapsto \bar B_i$ for every $i=2,\ldots,d$ (where $\bar X$ denotes the coset $X\cdot \mathrm{Z}(\dbU_4)\in\bar\dbU_4$, for $X\in\dbU_4$) yields a homomorphism $\bar\rho\colon G\to\bar\dbU_4$ which cannot lift to a homomorphism $\rho\colon G\to\dbU_4$ as $a_1b_1c_1\neq0$.
Thus, by Proposition~\ref{prop:masse unip} the Massey product $\langle\psi_1,\psi_2,\psi_3\rangle$
is defined but does not vanish.        \end{proof}

  \begin{prop}\label{prop:G2 massey 4}
   Let $G=\langle\:x_1,\ldots,x_d\:\mid \: r_0\:\rangle$ be a one-relator pro-$p$ group with $r_0$ of the first type as defined in Theorem~\ref{thm:main}.
   Then every defined 4-fold Massey product in $H^\bullet(G,\F_p)$ vanishes, unless $p=q=3$.
 \end{prop}

 \begin{proof}
Pick four elements $\psi_1,\psi_2,\psi_3,\psi_4\in H^1(G,\F_p)$ (by Proposition~\ref{prop:massey cup}--(i), we may assume that they are all non-trivial), and write
\[
 \psi_1=\sum_{i=1}^da_i\chi_i,\qquad \psi_1=\sum_{i=1}^db_i\chi_i,\qquad 
 \psi_3=\sum_{i=1}^dc_i\chi_i,\qquad \psi_4=\sum_{i=1}^dd_i\chi_i,\]
 with $a_i,b_i,c_i,d_i\in\F_p$.
 Then 
 \begin{equation}\label{eq:vanmassey G2 4}
    \psi_1\cup H^1(G,\F_p)+\psi_4\cup H^1(G,\F_p)=H^2(G,\F_p)\ni 0,
 \end{equation}
unless $a_i=d_i=0$ for every $i=2,\ldots,d$.
 
Suppose that the Massey product $\langle\psi_1,\psi_2,\psi_3,\psi_4\rangle$ is defined.
If some of the coefficients $a_i,d_i$ are not 0 for $i\in\{2,\ldots,d\}$, then by Proposition~\ref{prop:massey cup}--(iii) and by \eqref{eq:vanmassey G2 4} the Massey product $\langle\psi_1,\ldots,\psi_4\rangle$ vanishes.
Hence, we assume now that $\langle\psi_1,\psi_2,\psi_3,\psi_4\rangle$ is defined and that $a_i=d_i=0$ for every $i=2,\ldots,d$, while $a_1,d_1\neq0$.
By Proposition~\ref{prop:masse unip}--(i) one has 
\begin{equation}\label{eq:massey G2 cup}\begin{split}
0=\psi_2\cup\psi_3 &=\sum_{h=1}^{(d-1)/2}(b_{2h}c_{2h+1}-b_{2h+1}c_{2h})\chi_{2h}\cup\chi_{2h+1}\\
&=\left(\sum_{h=1}^{(d-1)/2}b_{2h}c_{2h+1}-b_{2h+1}c_{2h}\right)\chi_{2}\cup\chi_{3},
                                        \end{split}\end{equation}
and therefore $\sum_hb_{2h}c_{2h+1}-b_{2h+1}c_{2h}=0$ ---
while both cup-products $\psi_1\cup\psi_2$ and $\psi_3\cup\psi_4$ are trivial because $\chi_1\cup H^1(G,\F_p)=0$.
Henceforth we assume $q\neq3$, and we analyse the following two cases: (a) $n\geq 3$, and (b) $n=2$.

\noindent Case (a). Assume that $n\geq3$, and set 
\[\begin{split}
     A&=I_5+a_1E_{1,2}+b_1E_{2,3}+c_1E_{3,4}+d_1E_{4,5}\in\dbU_5,\\ B_i&=I_5+b_iE_{2,3}+c_iE_{3,4}\in\dbU_5
  \end{split}
\]
for $i=2,\ldots,d$.
By Lemma~\ref{lem:Um+1}--(ii), the exponent of $\dbU_5$ is $p^2$, if $p=2,3$, and $p$ if $p\geq5$.
Therefore, the exponent of $\dbU_5$ divides $q$, and $A^q=I_5$.
Moreover, one has $[A,_3B_2]=I_5$, and 
while $[B_{i},B_{i+1}]=I_5+(b_{i}c_{i+1}-b_{i+1}c_{i})E_{2,4}$, and thus
\[ \prod_{h=1}^{(d-1)/2}[B_{2h},B_{2h+1}]=I_5+E_{2,4}\cdot\sum_{h=1}^{(d-1)/2}b_{2h}c_{2h+1}-b_{2h+1}c_{2h}.
\]
Therefore, by \eqref{eq:massey G2 cup} one has
\begin{eqnarray}\label{eq:matrices relation G2 4fold}
  A^ q\cdot \left[A,_nB_2\right]\cdot\left[B_2,B_3\right]\cdots\left[B_{d-1},B_d\right] &=& I_5,\\
\text{and also }\qquad[[B_i,B_j],[B_{i'},B_{j'}]]&=& I_5   \label{eq:matrices relation G2 4fold bis}
\end{eqnarray}
for all $i,j,i',j'\in\{2,\ldots,d\}$
so that the assignment $x_1\mapsto A$, $x_i\mapsto B_i$ for $i=2,\ldots,d$ yields a homomorphism $\rho\colon G\to\dbU_5$.
 Hence, the Massey product $\langle\psi_1,\psi_2,\psi_3\rangle$ vanishes by Proposition~\ref{prop:masse unip}--(ii).
 
 \noindent Case (b). Assume that $n=2$, and let $A,B_i\in\dbU_5$ be as above.
 If $b_2=0$ or $c_2=0$, then one has $[A,_2B_2]=I_5$, and equality \eqref{eq:matrices relation G2 4fold} holds also with $n=2$.
Thus, the assignment $x_1\mapsto A$, $x_i\mapsto B_i$ for $i=2,\ldots,d$ yields a homomorphism $\rho\colon G\to\dbU_5$, and the Massey product $\langle\psi_1,\psi_2,\psi_3\rangle$ vanishes by Proposition~\ref{prop:masse unip}--(ii).
If $b_2\cdot c_2\neq 0$, then 
$$ [A,_2B_2]=I_5+a_1b_2c_2 E_{1,4}+b_2c_2d_1E_{2,5}\neq I_5.$$
Set $\tilde{B}_3=B_3-a_1b_2E_{1,3}+c_2d_1E_{3,4}$.
Then
\[\begin{split}
   [B_2,\tilde{B}_3]&= I_5-a_1b_2c_2 E_{1,4}+(b_2c_3-b_3c_2)E_{2,4}-b_2c_2d_1E_{2,5}\\&= [A,_2B_2]^{-1}+(b_2c_3-b_3c_2)E_{2,4},
  \end{split}\]
 and equality \eqref{eq:matrices relation G2 4fold} holds with $n=2$ and $\tilde{B}_3$ instead of $B_3$.
 Also equality \eqref{eq:matrices relation G2 4fold bis} holds after the replacement of $B_3$.
Thus, the assignment $x_1\mapsto A$, $x_3\mapsto\tilde{B}_3$, $x_i\mapsto B_i$ for $i=2,4,\ldots,d$, yields a homomorphism $\rho\colon G\to\dbU_5$, and the Massey product $\langle\psi_1,\psi_2,\psi_3,\psi_4\rangle$ vanishes by Proposition~\ref{prop:masse unip}--(ii).
        \end{proof}


\subsection{The second type of $r_0$}\label{ssec:masseyG1}

As we did for the first family of pro-$p$ groups, we split the proof of Theorem~\ref{prop:intro}--(iib) in the following two propositions.
 
 \begin{prop}\label{prop:G1 massey 3}
  Let $G=\langle\:x_1,\ldots,x_d\:\mid \: r_0\:\rangle$ be a one-relator pro-$p$ group with $r_0$ of the second type as defined in Theorem~\ref{thm:main}.
 Then every defined 3-fold Massey product in $H^\bullet(G,\F_p)$ vanishes if and only if $q\neq3$ and $n\geq3$.
 \end{prop}
 
 \begin{proof}
Pick three elements $\psi_1,\psi_2,\psi_3\in H^1(G,\F_p)$ (by Proposition~\ref{prop:massey cup}--(i), we may assume that they are all non-trivial), and write
\[
 \psi_1=\sum_{i=1}^da_i\chi_i,\qquad \psi_1=\sum_{i=1}^db_i\chi_i,\qquad \psi_3=\sum_{i=1}^dc_i\chi_i,\]
 with $a_i,b_i,c_i\in\F_p$.
 Then by Proposition~\ref{prop:onerel} one has 
 \begin{equation}\label{eq:vanmassey G1 3}
    \psi_1\cup H^1(G,\F_p)+\psi_3\cup H^1(G,\F_p)=H^2(G,\F_p)\ni 0,
 \end{equation}
unless $a_i=c_i=0$ for every $i=3,\ldots,d$.
 
Suppose that the 3-fold Massey product $\langle\psi_1,\psi_2,\psi_3\rangle$ is defined.
If some of the coefficients $a_i,c_i$ are not 0 for $i\in\{3,\ldots,d\}$, then by Proposition~\ref{prop:massey cup}--(iii) and by \eqref{eq:vanmassey G1 3} the Massey product $\langle\psi_1,\psi_2,\psi_3\rangle$ vanishes.

Assume that $q\neq3$, $n\geq3$, and let $\psi_1,\psi_3$ be such that $a_i=c_i=0$ for every $i=3,\ldots,d$. 
Set
\[\begin{split}
   A_i=I_4+a_iE_{1,2}+b_iE_{2,3}+c_iE_{3,4}\in\dbU_4,\qquad&\text{for }i=1,2, \\
   B_i=I_4+b_iE_{2,3}\in\dbU_4,\qquad&\text{for }i=3,\ldots,d.
  \end{split}\]
By Lemma~\ref{lem:Um+1}--(ii), the exponent of $\dbU_4$ is $p^2$, if $p=2,3$, and $p$ if $p\geq5$.
Therefore, the exponent of $\dbU_4$ divides $q$, and $A_1^q=I_4$.
Also, by Lemma~\ref{lem:Um+1}--(c) one has $\gamma_4(\dbU_4)=\{1\}$.
Finally, one has $[B_i,B_{i+1}]=I_4$ for every $3\leq i\leq d-1$, $i$ odd, and $[A_1,_3A_2]=I_3$.
Therefore, if $n\geq3$ then
\begin{equation}\label{eq:rho r G1}
 A^ q\cdot \left[A_1,_nA_2\right]\cdot\left[B_3,B_4\right]\cdots\left[B_{d-1},B_d\right]=I_4.
 \end{equation}
and the assignment $x_i\mapsto A_i$ and $x_j\mapsto B_j$ for $i=1,2$ and $j=3,\ldots,d$ yields a homomorphism $\rho\colon G\to\dbU_4$.
Thus, $\langle\psi_1,\psi_2,\psi_3\rangle$ vanishes by Proposition~\ref{prop:masse unip}--(ii).

Conversely, assume that $q=3$ or $n=2$, and let $\psi_1,\psi_3$ be such that $a_i=c_i=0$ for $i=3,\ldots,d$.
For every matrix $A_i=(a(i)_{h,k})\in\dbU_4$, with $i=1,2$, and $B_j=(b(j)_{h,k})\in\dbU_4$, with $j=3,\ldots,d$, such that $a(i)_{1,2}=a_i$, $a_{2,3}(i)=b_i$, $a(i)_{3,4}=c_i$, and $b(j)_{2,3}=b_j$, $b(j)_{1,2}=b(i)_{3,4}=0$ for all $i,j$, 
one has 
\[\begin{split}
[B_{j},B_{j'}]&=I_4,\qquad\text{for every}j,j'\in\{3,\ldots,d\},\\
[A_1,_2A_2]&=I_4+(a_1b_2c_2-2a_2b_1c_2+a_2b_2c_1)E_{1,4},\\
A_1^3&=I_4+a_1b_1c_1E_{1,4}.
  \end{split}\]
In this case, equality \eqref{eq:rho r G2} holds modulo $\mathrm{Z}(\dbU_4)$, and the assignment $x_i\mapsto \bar A_i$ and $x_j\mapsto \bar B_j$ for every $i=1,2$, $j=3,\ldots,d$ (where $\bar X$ denotes the coset $X\cdot \mathrm{Z}(\dbU_4)\in\bar\dbU_4$, for $X\in\dbU_4$) yields a homomorphism $\bar\rho\colon G\to\bar\dbU_4$, and therefore by Proposition~\ref{prop:masse unip}--(i) the Massey product $\langle\psi_1,\psi_2,\psi_3\rangle$ is defined.
One has two cases: 
\begin{itemize}
 \item[(a)] if $q=3$ and $a_1b_1c_1\neq 0$ (and moreover $a_1b_2c_2-2a_2b_1c_2+a_2b_2c_1\neq -a_1b_1c_1$ if $n=2$), then $\bar\rho$ cannot lift to a homomorphism $\rho\colon G\to \dbU_4$;
 \item[(b)] similarly, if $n=2$ and $a_1b_2c_2-2a_2b_1c_2+a_2b_2c_1\neq0$ (and moreover $a_1b_1c_1\neq -a_1b_2c_2+2a_2b_1c_2-a_2b_2c_1$ if $q=3$), then $\bar\rho$ cannot lift to a homomorphism $\rho\colon G\to \dbU_4$.
\end{itemize}
In both cases, Proposition~\ref{prop:masse unip}--(ii) implies that $\langle\psi_1,\psi_2,\psi_3\rangle$ does not vanish.
     \end{proof}

\begin{rem}\label{rem:Massey 3 G1 MT}\rm
 If $n=2$, then one may deduce that $H^\bullet(G,\F_p)$ has defined Massey products which do not vanish also from \cite[Thm.~7.12]{MT:masseyGal}.
\end{rem}

 \begin{prop}\label{prop:G1 massey 4}
 Let $G=\langle\:x_1,\ldots,x_d\:\mid \: r_0\:\rangle$ be a one-relator pro-$p$ group with $r_0$ of the second type as defined in Theorem~\ref{thm:main}.
 If $q\neq3$ and $n\geq4$, then every defined 4-fold Massey product in $H^\bullet(G,\F_p)$ vanishes.
 \end{prop}

  \begin{proof}
Pick four non-trivial elements $\psi_1,\psi_2,\psi_3,\psi_4\in H^1(G,\F_p)$, and write
\[
 \psi_1=\sum_{i=1}^da_i\chi_i,\qquad \psi_1=\sum_{i=1}^db_i\chi_i,\qquad 
 \psi_3=\sum_{i=1}^dc_i\chi_i,\qquad \psi_4=\sum_{i=1}^dd_i\chi_i,\]
 with $a_i,b_i,c_i,d_i\in\F_p$.
 Then 
 \begin{equation}\label{eq:vanmassey G1 4}
    \psi_1\cup H^1(G,\F_p)+\psi_4\cup H^1(G,\F_p)=H^2(G,\F_p)\ni 0,
 \end{equation}
unless $a_i=d_i=0$ for every $i=3,\ldots,d$.
 
Suppose that the Massey product $\langle\psi_1,\psi_2,\psi_3,\psi_4\rangle$ is defined.
If some of the coefficients $a_i,d_i$ are not 0 for $i\in\{3,\ldots,d\}$, then by Proposition~\ref{prop:massey cup}--(iii) and by \eqref{eq:vanmassey G1 4} the Massey product $\langle\psi_1,\ldots,\psi_4\rangle$ vanishes.
Hence, we assume now that $\langle\psi_1,\psi_2,\psi_3,\psi_4\rangle$ is defined and that $a_i=d_i=0$ for every $i=3,\ldots,d$.

By Proposition~\ref{prop:masse unip}--(i) one has 
\begin{equation}\label{eq:massey G1 cup}\begin{split}
0=\psi_2\cup\psi_3 &=\sum_{h=2}^{d/2}(b_{2h-1}c_{2h}-b_{2h}c_{2h-1})\chi_{2h-1}\cup\chi_{2h}\\
&=\left(\sum_{h=1}^{(d-1)/2}b_{2h}c_{2h+1}-b_{2h}c_{2h-1}\right)\chi_{3}\cup\chi_{4},
                                        \end{split}\end{equation}
and therefore $\sum_hb_{2h-1}c_{2h}-b_{2h}c_{2h-1}=0$ ---
while both cup-products $\psi_1\cup\psi_2$ and $\psi_3\cup\psi_4$ are trivial because $\chi_1\cup H^1(G,\F_p)=\chi_2\cup H^1(G,\F_p)=0$.

Assume that $n\geq4$, and set 
\[\begin{split}
     A_i&=I_4+a_iE_{1,2}+b_iE_{2,3}+c_iE_{3,4}+d_iE_{4,5}\in\dbU_5,\\ 
     B_j&=I_4+b_jE_{2,3}+c_jE_{3,4}\in\dbU_5
  \end{split}
\]
for $i=1,2$ and $j=3,\ldots,d$.
By Lemma~\ref{lem:Um+1}--(ii), the exponent of $\dbU_5$ is $p^2$, if $p=2,3$, and $p$ if $p\geq5$.
Therefore, the exponent of $\dbU_5$ divides $q$, and $A_1^q=I_5$.
Moreover, one has $[A_1,_4A_2]=I_5$ and $[B_{j},B_{j+1}]=I_5+(b_{j}c_{j+1}-b_{j+1}c_{j})E_{2,4}$, so that
\[ \prod_{h=2}^{d/2}[B_{2h-1},B_{2h}]=I_5+E_{2,4}\cdot\sum_{h=2}^{d/2}b_{2h-1}c_{2h}-b_{2h}c_{2h-1}.
\]
Therefore, by \eqref{eq:massey G1 cup} one has
\begin{eqnarray}\label{eq:matrices relation G1 4fold}
 A^ q\cdot \left[A_1,_nA_2\right]\cdot\left[B_3,B_4\right]\cdots\left[B_{d-1},B_d\right]&=&I_5,\\
 \text{and also } \qquad[[B_i,B_j],[B_{i'},B_{j'}]]&=& I_5\label{eq:matrices relation G1 4fold bis}
\end{eqnarray}
for every $i,j,i',j'\in\{3,\ldots,d\}$, so that the assignment $x_i\mapsto A_i$, $x_j\mapsto B_j$ for $i=1,2$, $j=3,\ldots,d$, yields a homomorphism $\rho\colon G\to\dbU_5$.
 Hence, the Massey product $\langle\psi_1,\psi_2,\psi_3,\psi_4\rangle$ vanishes by Proposition~\ref{prop:masse unip}--(ii).
\end{proof}

\begin{rem}\rm
If $n=3$, then the vanishing of 4-fold Massey products depends also on $d$.
For example, the Massey product $ \langle \chi_1, \chi_2+\chi_{3},\chi_2+\chi_{3} , \chi_2\rangle$
is always defined, and does not vanish if $d=4$, while it vanishes if $d\geq 6$.
\end{rem}

We conclude with the following observation.

\begin{rem}\label{rem:final}\rm
From the two families of pro-$p$ groups defined in Theorem~\ref{thm:main}, one may produce several further examples of pro-$p$ groups with more than one relation, whose $\F_p$-cohomology algebra is quadratic, for example among {\sl mild} pro-$p$ groups ---
e.g., if $$G=\langle\:x_1,\ldots,x_d\:\mid\:r_0,r_1\:\rangle,\qquad \text{with }r_1=[x_{d-2},x_{d-1}]\cdot\prod_{i,j}[x_i,x_j]$$ for some $3\leq i<j\leq d$, then $G$ is mild and $H^\bullet(G,\F_p)$ is quadratic, cf. \cite[Prop.~3.6]{cq:2rel}.
It may be interesting to investigate the behavior of these pro-$p$ groups with respect to 3- and 4-fold Massey products and free and abelian subgroups, in order to obtain results similar to Theorem~\ref{prop:intro}--(iia)--(iib) and Proposition~\ref{prop:BK split}.
\end{rem}


\subsection*{Acknowledgement}
{\small 
The author is grateful to: N.D.~T\^an and J.~Minac, for the inspiring discussions on J.P.~Labute's article \cite{labute:demushkin}; I.~Efrat, for the investigation on the Kummerian property of pro-$p$ groups (a fundamental ingredient for this paper) pursued in \cite{eq:kummer}; I.~Snopce, for the stimulating discussions on pro-$p$ groups and absolute Galois groups of fields; F.~Matucci and Th.S.~Weigel, for their useful comments; and the anonymous referee, for her/his careful review and suggestions.
Moreover, the author wishes to remark that the research carried out in the current work (and also in past papers) has been deeply influenced by J.P.~Labute's life-time work on pro-$p$ groups and their cohomology.}

\end{document}